\documentclass[reqno,11pt, a4] {amsart}

\usepackage[active]{srcltx}

\usepackage[usenames,dvipsnames,svgnames,table]{xcolor}

\usepackage{pb-diagram}

\usepackage{mathrsfs}
\usepackage{amsmath}
\usepackage{amssymb}
\usepackage{amscd}
\usepackage{amsthm}
\usepackage[latin1]{inputenc}
\usepackage{graphics}
\usepackage{varioref}

 \newcommand{\changed}[1]{#1}%\textcolor{Red}{#1}}

\labelformat{enumi}{(#1)}

\newtheorem{teo}[equation]{Theorem}
\newtheorem{defin}[equation]{Definition}
\newtheorem{remark}[equation]{Remark}
\newtheorem{prop}[equation]{Proposition}
\newtheorem{cor}[equation]{Corollary}
\newtheorem{lemma}[equation]{Lemma}
\newtheorem{Example}[equation]{Example}

\newcommand{\Keler}             {K\"{a}hler }

\newcommand{\lds}{\ldots}
\newcommand{\cds}{\cdots}
\newcommand{\cd}{\cdot}
\renewcommand{\setminus}{-}
\newcommand{\om}{\omega}

\renewcommand{\phi}{\varphi}

\newcommand{\Tr}{\operatorname{Tr}}
\newcommand{\ra}{\rightarrow}

\newcommand{\C}{\mathbb{C}}
\newcommand{\R}{\mathbb{R}}

\newcommand{\Lie}{\operatorname{Lie}}
\newcommand{\SU} {\operatorname{SU}}
\newcommand{\su} {\mathfrak{su}}
\newcommand{\Sl}{\operatorname{SL}}
\newcommand{\Gl}{\operatorname{GL}}

\newcommand{\gl}{\operatorname{\mathfrak{gl}}}
\newcommand{\restr}[1]          {\vert_{#1}}

\newcommand{\Ad}{\operatorname{Ad}}
\newcommand{\ad}{{\operatorname{ad}}}

\newcommand{\ga}{\gamma}

\newcommand{\meno}{^{-1}}
\newcommand{\Zeta}{{\mathbb{Z}}}

\newcommand{\PP}{\mathbb{P}}

\newcommand{\enf}{\emph}

\newcommand{\desudt}[1] []      {\dfrac {\mathrm {d} #1 }{\mathrm {dt}}}
\newcommand{\desudtzero}        {\desudt \bigg \vert _{t=0} }

\renewcommand{\root}{\Delta}
\newcommand{\conv} {\operatorname{conv}}
\newcommand{\liu}{\mathfrak{u}}

\newcommand{\lieh}{\mathfrak{h}}
\newcommand{\liek}{\mathfrak{k}}

\newcommand{\liel}{\mathfrak{l}}
\newcommand{\lieg}{\mathfrak{g}}

\newcommand{\lieb}{\mathfrak{b}}
\newcommand{\liep}{\mathfrak{p}}

\newcommand{\liez}{\mathfrak{z}}
\newcommand{\lies}{\mathfrak{s}}

\newcommand{\liet}{\mathfrak{t}}

\newcommand{\grass}{\operatorname{\mathbb{G}}}
\newcommand{\chern}{\operatorname{c}}
\newcommand{\im}{\operatorname{Im}}

\newcommand{\la}{\lambda}

\newcommand{\alfa}{\alpha}

\newcommand{\vacuo}{\emptyset}
\newcommand{\OO}{\mathcal{O}}               % coadjoint orbit
\renewcommand{\c}{{\widehat{\OO}}}          % convex envelope
\newcommand{\cp}{{\widehat{\OO'}}}          % convex envelope
\newcommand{\polp}{{P}}                     % momentum polytope
\newcommand{\ml}{\operatorname{Max}}        % set of maximum point of a component of the momentum map
\newcommand{\ext}{\operatorname{ext}}       % set of extreme points of a convex set
\newcommand{\sx}{\langle}                   % scalar product
\newcommand{\xs}{\rangle}
\newcommand{\relint}{\operatorname{relint}} % relative interior of a convex body
\newcommand{\Weyl}{W}                       % Weyl group
\newcommand{\cchamber}{\overline{C^+}}
\newcommand{\chamber}{{C^+}}
\newcommand{\Crit}{\operatorname{Crit}}     % set of critical points
\newcommand{\roots}{\Delta}                 % set of roots
\newcommand{\simple}{\Pi}                   % set of simple roots
\newcommand{\mom}{\Phi}                     % momentum map
\newcommand{\momt}{\Phi_T}                  % momentum map for the torus
\newcommand{\faces}{\mathscr{F}(\c)}       % faces of \OO
\newcommand{\facesp}{\mathscr{F}(\polp)}    % faces of P that contain \x
\newcommand{\E}{{\mathcal{E}_F}}               % associated bundle
\newcommand{\Fo}{\relint F}
\newcommand{\V}{V}

\newcommand{\spam}{\operatorname{span}}

\newcommand{\hilbo}{\mathcal{H}}

\newcommand{\CF}{C_F}

\newcommand{\aff}{  \operatorname{aff}}
\newcommand{\diag}{\operatorname{diag}}

\newcommand{\SB}{\mathcal{S}_B}

\begin{document}

\title{Coadjoint orbitopes}

\author{Leonardo Biliotti}

\author{Alessandro Ghigi}

\author{Peter Heinzner}

%\today

\begin{abstract}
  We study \emph{coadjoint orbitopes}, i.e. convex hulls of coadjoint
  orbits of compact Lie groups. We show that up to conjugation the
  faces are completely determined by the geomety of the faces of the
  convex hull of Weyl group orbits.  We also consider the geometry of
  the faces and show that they are themselves coadjoint orbitopes.
  From the complex geometric point of view the sets of extreme points
  of a face are realized as compact orbits of parabolic subgroups of
  the complexified group.
\end{abstract}

  \address{Universit\`{a} di Parma} \email{leonardo.biliotti@unipr.it}
  \address{Universit\`a di Milano Bicocca}
  \email{alessandro.ghigi@unimib.it}

  \address{Ruhr Universit\"at Bochum} \email{peter.heinzner@rub.de}

  \thanks{The first author was partially supported by GNSAGA of INdAM.
    The second author was partially supported by GNSAGA of INdAM and
    by PRIN 2009 MIUR ''Moduli, strutture geometriche e loro
    applicazioni''. The third author was partially supported by
    DFG-priority program SPP 1388 (Darstellungstheorie)}


  \maketitle

\tableofcontents{}

\section*{Introduction}

Let $K$ be a compact Lie group and let $K \ra \Gl(V)$ be a
finite-dimensional representation.  An \emph{orbitope} is by
definition the convex envelope of an orbit of $K$ in $V$ (see
\cite{sanyal-sottile-sturmfels-orbitopes}).  An interesting class of
orbitopes is given by the convex envelope of coadjoint orbits.  We
call these \emph{coadjoint orbitopes}. The case of an integral orbit
has been studied in \cite{biliotti-ghigi-2}, where it was realised
that a remarkable construction introduced by Bourguignon, Li and Yau
\cite{bourguignon-li-yau} in the case of complex projective space can
be generalized to arbitrary flag manifolds.  This allowed to show that
the convex envelope of an integral coadjoint orbit is equivariantly
homeomorphic to a Satake-Furstenberg compactification.  This
homeomorphism is constructed by integrating the momentum map, but
unfortunately it is not explicit and its nature is not yet
well-understood. On the other hand, the Satake-Furstenberg
compactifications admit a very precise combinatorial description going
back to Satake \cite{satake-compactifications}.

The aim of this paper is to give a precise description of
the boundary structure of coadjoint orbitopes without the integrality
assumption and without relying on the connection with
Satake-Furstenberg compactifications.
% In the second place we establish
% the combinatorial description of the orbitope in terms of the notions
% put forward by Satake to describe his compactifications. This
% description is established directly, without need to mention the
% compactifications, and is valid for all orbits integral or not.

To a coadjoint orbit $\OO$ we associate its convex hull $\c$. The aim
is to describe the \emph{faces} of $\c$ and their extremal points in
the sense of convex geometry. If we fix a maximal torus $T$, there is
another convex set associated to $\OO$, namely the Kostant polytope
$P$, which is the convex hull of a Weyl group orbit in $\liet$.
Denote by $\faces$ the faces of $\OO$ and by $\facesp$ the faces of
$P$. $K$ acts on $\faces$ and the Weyl group $W$ acts on $\facesp$.
 In \S \ref{relazione-politopo} we show the following.
\begin{teo}
  \label{main}
  If \changed{$\sigma \in \facesp$} and $\sigma^\perp$ is the set of vectors in
  $\liet$ which are orthogonal to $\sigma$, then $Z_K(\sigma^\perp)\cd
  \sigma$ is a face of $\c$. Moreover the map $\sigma \mapsto
  Z_K(\sigma^\perp)\cd \sigma$ passes to the quotients and the
  resulting map $\facesp/W \ra \faces /K$  is a bijection.
\end{teo}

During the proof of Theorem \ref{main} we show that every face is
\emph{exposed} (see Definition \ref{def-exposed}).  The extremal
points of an exposed face form a symplectic submanifold of $\OO$, that
has been studied since the important work of Duistermaat, Kolk,
Varadarajan and Heckman \cite{duistermaat-kolk-varadarajan,
  heckman-thesis}. In \S \ref{group} we reformulate their results to
describe the structure of exposed faces using the momentum map.  It
follows that every face is itself a coadjoint orbitope (Theorem
\ref{facciona-orbita}) and that it is stable under a maximal torus
(Theorem \ref{torus-preserves}).  For $K=\operatorname{SO}(n)$ a proof
of Theorem \ref{main} is given in \cite [\S3.2]
{sanyal-sottile-sturmfels-orbitopes}. Their proof relies on the
representation of these orbitopes as spectrahedra.

The second main result of the paper deals with the complex geometry of
$\OO$.  Consider the K\"{a}hler structure on $\OO$ and the holomorphic
action of $G = K^\C$ (see \S \ref {pre-orbits}).
\begin{teo}
  \label{main-2}
  If $F$ is a face of $\c$, then $\ext F \subset \OO$ is a closed
  orbit of a parabolic subgroup of $ G$. Conversely, if $P\subset G$
  is a parabolic subgroup, then it has a unique closed orbit
  $\OO'\subset \OO$ and there is a face $F$ such that $\ext F = \OO'$.
\end{teo}

In \S \ref{strata} we show that there is a finite stratification of
the boundary of $\c$ in terms of face types, where the strata are
smooth fibre bundles over flag manifolds.  In \S \ref{satake-section}
we give a description of the faces in terms of root data, using the
formalism of $x$-connected subset of simple roots developed by Satake
\cite{satake-compactifications}.  In the last section we prove that if
$\OO$ is an integral orbit (i.e. it corresponds to a representation),
the same holds for $\ext F$ for any faces $F\subset \c$.

\medskip

We think that many other aspects of these orbitopes are worth
studying.  It would be interesting to find explicity formulae for the
volume, the surface area and the Quermassintegrals.  Also, in a future
paper we plan to study the following class of orbitopes: $G$ is a real
semisimple Lie group with Cartan decomposition $\lieg=\liek\oplus
\liep$, $\OO$ is a $K$-orbit in $\liep$ and $\c$ is the convex hull of
$\OO$.  Coadjoint orbitopes correspond to the special case where
$G=K^\C$ and $\liep = i \liek$.

{\bfseries \noindent{Acknowledgements.}}  The first two authors wish
to thank Laura Geatti for interesting discussions and the Fakultät für
Mathematik of Ruhr-Universität Bochum for hospitality.

\section{Preliminaries from convex geometry}
\label{pre-convex}

It is useful to recall a few definitions and results regarding convex
sets (see e.g. \cite{schneider-convex-bodies}).  Let $V$ be a real
vector space and $E\subset V$ a convex subset.  The \emph{relative
  interior} of $E$, denoted $\relint E$, is the interior of $E$ in its
affine hull.  \changed{ If $x, y\in E$, then $[x,y]$ denotes the
  closed segment from $x$ to $y$, i.e. $[x,y]:= \{ (1-t)x + t y : t\in
  [0,1]\}$. } A face $F$ of $E$ is a convex subset $F\subset E$ with
the following property: if $x,y\in E$ and $\relint[x,y]\cap F\neq
\vacuo$, then $[x,y]\subset F$.  The \emph{extreme points} of $E$ are
the points $x\in E$ such that $\{x\}$ is a face.  If $E$ is compact
the faces are closed \cite[p. 62]{schneider-convex-bodies}.  If $F$ is
a face of $E$ we say that $\relint F$ is an \emph{open face} of $E$.
A face distinct from $E$ and $\vacuo$ will be called a \enf{proper
  face}.

Assume for simplicity that a scalar product $\sx \ , \, \xs$ is fixed
on $V$ and that $E \subset V$ is a \enf{compact} convex subset with
nonempty interior.

\begin{defin}\label{def-exposed}
  The \enf{support function} of $E$ is the function
  \begin{gather}
    \label{def-support}
    h_E : V \ra \R \qquad h_E(u) = \max_{x \in E} \sx x, u \xs.
  \end{gather}
  If $ u \neq 0$, the hyperplane $H(E, u) : = \{ x\in V : \sx x, u \xs
  = h_E(u)\}$ is called the \enf{supporting hyperplane} of $E$ for
  $u$. The set
  \begin{gather*}
    F_u (E) : = E \cap H(E,u)
  \end{gather*}
  is a face and it is called the \enf{exposed face} of $E$ defined by
  $u$ or also the \emph{support set} of $E$ for $u$.
\end{defin}
In using the notation $F_u(E)$ we will tacitly assume that the affine
span of $E$ is $V$. Hence by definition an exposed face is proper.  We
notice that in general not all faces of a convex subsets are exposed.
A simple example is given by the convex hull of a closed disc and a
point outside the disc: the resulting convex set is the union of the
disc and a triangle. The two vertices of the triangle that lie on the
boundary of the disc are non-exposed 0-faces.

\begin{lemma}\label{ext-facce}
  If $F$ is a face of a convex set $E$, then $\ext F = F \cap \ext E$.
\end{lemma}
\begin{proof}
  It is immediate that $F \cap \ext E \subset \ext F$. The converse
  follows from the definition of a face.
\end{proof}

\begin{lemma}
  \label{convex-orbit}
  If $G$ is a compact group, $V$ is a representation space of $G$ and
  $G\cd x$ is an orbit of $G$, then $\conv ( G\cd x)$ contains a fixed
  point of $G$. Moreover any fixed point contained in $\conv (G\cd x)$
  lies in $\relint \conv (G\cd x)$.
\end{lemma}
\begin{proof}
 \changed{ Set}
  \begin{gather*}
    \bar{x}:= \int_G g\cd x \, dg
  \end{gather*}
  where $dg$ is the normalized Haar measure. Then $\bar{x}$ is
  $G$-invariant and belongs \changed{$\conv (G \cd x)$}.  Now let $y$ be any fixed point of
  $G$ that lies in \changed{$\conv (G \cd x)$}.  By Theorem \ref{schneider-facce} there is a
  unique face \changed{$F \subset \conv (G \cd x)$} such that $y$ belongs to $\relint
  F$. Since \changed{$\conv(G \cd x)$} is $G$-invariant and $ y $ is fixed by \changed {$G$}, it
  follows that $a\cdot F =F$ for any $a\in G$. So $F$ is
  $G$-invariant, and hence also $\ext F$ is $G$-invariant. Since $\ext
  F \subset \ext (\conv (G\cd x) ) \subset G\cd x$, it follows that
  $\ext F = G\cd x$ and hence that \changed{$F = \conv (G\cd x)$}.
\end{proof}

\begin{prop}
  \label{u-cono}
  If $F \subset E$ is an exposed face, the set $\CF : = \{ u\in V:
  F=F_u(E) \}$ is a convex cone. If $G$ is a compact subgroup of
  $O(V)$ that preserves both $E$ and $F$, then $\CF$ contains a fixed
  point of $G$.
\end{prop}
\begin{proof}
  Let $u_1, u_2 \in \CF$ and $\la_1 , \la_2 \geq 0$ and set $u = \la_1
  u_1 + \la_2 u_2$. We need to prove that if at least one of $\la_1,
  \la_2$ is strictly positive, then $F=F_u(E)$.  Assume for example
  that $\la_1 >0$. It is clear that $h_E (u) \leq \la_1 h_E(u_1 ) +
  \la_2 h_E(u_2)$.  If $x\in F$, then
  \begin{gather*}
    \sx x, u\xs = \la_1 \sx x, u_1 \xs + \la_2 \sx x, u_2\xs = \la_1
    h_E(u_1 ) + \la_2 h_E(u_2).
  \end{gather*}
  Hence $h_E (u) = \la_1 h_E(u_1 ) + \la_2 h_E(u_2)$ and $F\subset
  F_u(E)$. Conversely, if $x\in F_u (E)$, then
  \begin{gather*}
    0= h_E(u) -\sx x, u\xs = \la_1 ( h_E(u_1) - \sx x, u_1 \xs ) +
    \la_2 ( h_E(u_2) - \sx x, u_2\xs).
  \end{gather*}
  Since $\la_1 >0$ we get $h_E(u_1) - \sx x, u_1 \xs = 0$, so $x\in
  F_{u_1}(E) = F$.  Thus $F= F_u(E)$. This proves the first fact.  To
  prove the second, pick any vector $u\in \CF$ and apply the previous
  lemma to the orbit $G\cd u \subset \CF$: this yields a $G$-invariant
  $\bar{u} \in \CF$.
\end{proof}

\begin{teo} [\protect{\cite[p. 62]{schneider-convex-bodies}}]
  \label{schneider-facce} If $E$ is a compact convex set and $F_1,F_2$
  are distinct faces of $E$ then $\relint F_1 \cap \relint
  F_2=\vacuo$. If $G$ is a nonempty convex subset of $ E$ which is
  open in its affine hull, then $G \subset\relint F$ for some face $F$
  of $E$. Therefore $E$ is the disjoint union of its open faces.
\end{teo}

\begin{lemma}
  If $E$ is a compact convex set and $F\subsetneq E$ is a face, then
  $\dim F < \dim E$.
\end{lemma}
\begin{proof}
  If $\dim F = \dim E$, then $\relint F$ is open in the affine span of
  $E$, so $\relint F \subset \relint E$. By the previous theorem this
  implies that $F=E$.
\end{proof}

\begin{lemma}
  \label{face-chain}
  If $E$ is a compact convex set and $F\subset E$ is a face, then
  there is a chain of faces
  \begin{gather*}
    F_0=F \subsetneq F_1 \subsetneq \cds \subsetneq F_k=E
  \end{gather*}
  which is maximal, in the sense that for any $i$ there is no face of
  $E$ strictly contained between $F_{i-1}$ and $F_i$.
\end{lemma}
\begin{proof}
  If $F=E$ there is nothing to prove. Otherwise put $F_0:=F$. If there
  is no face strictly contained between $F_0$ and $E$, just set
  $F_1=E$. Otherwise we find a chain $F_0 \subsetneq F_1 \subsetneq
  F_2=E$. If this is not maximal, we can refine it. Repeting this step
  we get a chain with $k+1$ elements. Since $\dim F_{i-1} < \dim F_i$,
  $k\leq n$. Therefore the chain gotten after at most $n$ steps is
  maximal.
\end{proof}

\begin{lemma} \label{micro-convesso} If $E$ is a convex subset of
  $\R^n$, $M\subset \R^n$ is an affine subspace and $F\subset E$ is a
  face, then $F\cap M $ is a face of $E\cap M$.
\end{lemma}
\begin{proof}
  If $x,y \in E\cap M$ and $\relint [x,y] \cap F\cap M \neq \vacuo$
  then $[x, y] \subset F$ since $F$ is a face, but $[x,y]$ is also
  contained in $M$ since $M$ is affine. So $[x,y] \subset F\cap M$ as
  desired.
\end{proof}

\section{Coadjoint orbits}
\label{pre-orbits}

Throught the paper we will use the following notation.  $K$ denotes a
compact connected semisimple Lie group with Lie algebra $\liek$.  If
$T\subset K$ is a maximal torus and $\simple \subset \root(\liek^\C,
\liet^C)$ is a set of simple roots, the Weyl chamber of $\liet$
corresponding to $\simple$ is defined by
\begin{gather*}
  \chamber :=\{v\in \liet : -i\alfa(v) >0 \text { for any } \alfa \in
  \root_+\}.
\end{gather*}
$B$ is the Killing form of $\liek^\C$ and $\sx\, , \xs =
-B\restr{\liek \times \liek}$ is a scalar product on $\liek$.  By
means of $\sx \, , \xs$ we identify $\liek $ with $\liek^*$.

\begin{lemma} \label{knappone} Let $T\subset K$ be a maximal torus,
  let $\roots$ be the root system of $( \liek^\C, \liet^C)$ and let
  $\simple \subset \root$ be a base.  Define $H_\alfa \in \liet^\C$ by
  the formula $ B(H_\alfa , \cd ) = \alfa(\cd )$ and choose a nonzero
  vector $X_\alfa \in \lieg_\alfa$ for any $\alfa \in \root$.  For $
  \alfa \in \root_+$ set
  \begin{gather*}
    \label{base-compatta}
    u_\alfa:=\frac{1}{\sqrt{2}} (X_\alfa - X_{-\alfa}) \qquad v_\alfa:
    = \frac{i}{\sqrt{2}} (X_\alfa + X_{-\alfa}).
  \end{gather*}
  Then it is possible to choose the vectors $X_\alfa$ in such a way
  that $ [X_\alfa , X_{-\alfa} ] = H_\alfa$ and so that the set
  $\{u_\alfa, v_\alfa | \alfa \in \root_+\}$ be orthonormal with
  respect to $\sx \ ,\ \xs = - B$.  Moreover for $y\in \liet$
  \begin{gather*}
    [y, u_\alfa ] = - i \alfa(y) v_\alfa \qquad [y, v_\alfa ] =
    i\alfa(y) u_\alfa\qquad [u_\alfa,v_\alfa ] = i H_\alfa.
  \end{gather*}
\end{lemma}
For a proof see e.g. \cite[pp. 353-354]{knapp-beyond}.  Set
\begin{gather}
  \label{Zalfa}
  Z_\alfa = \R u_\alfa \oplus \R v_\alfa.
\end{gather}
Then
\begin{gather*}
  \liek = \liet \oplus \bigoplus_{\alfa \in \roots_+} Z_\alfa .
\end{gather*}

If $\OO$ is an adjoint orbit of $K$ and $x\in \OO$, then
\begin{gather*}
  \qquad T_{x} \OO = \im \ad \ x = \bigoplus_{ \alfa \in E} Z_\alfa
\end{gather*}
where $ E : = \{\alfa \in \root_+ : \alfa(x) \neq 0\}$.  Denote by
$v_\OO$ the vector field on $\OO$ defined by $v\in \liek$.
Explicitely $v_\OO (x) = [v,x]$.  Since we identify $\liek \cong
\liek^*$ we may regard $\OO$ as a coadjoint orbit. As such it is
equipped with a $K$-invariant symplectic form $\om$, named after
Kostant, Kirillov and Souriau, and defined by the following rule.  For
$u, v \in \liek$
\begin{gather*}
  \om_x \left ( u_\OO \left (x \right ), v_\OO \left (x \right)\right) := \sx x, [u,v]\xs.
\end{gather*}
See e.g. \cite[p. 5] {kirillov-lectures}.  $\om$ is a $K$-invariant
symplectic form on $\OO$ and the inclusion $\OO \hookrightarrow \liek$
is the momentum map.

\label{momentum-notation-page}
If $T\subset K$ is a maximal torus, we denote by $W(K,T)$ or simply by
$W$ the Weyl group of $(K,T)$. We let $\pi: \liek \ra \liet$ denote
the orthogonal projection with respect to the scalar product $\sx \ ,
\ \xs= - B$. Its restriction to $\OO$ is denoted by $\momt : \OO \ra
\liet$; it is the momentum map for the $T$-action on $\OO$.  $\polp: =
\momt(\OO)$ is the momentum polytope.  The following convexity theorem
of Kostant \cite{kostant-convexity} is the basic ingredient in the
whole theory.
\begin{teo}[Kostant] \label{Kostant} Let $K$ be a compact connected
  Lie group, let $T \subset K$ be a maximal torus and let $\OO$ be a
  coadjoint orbit. Then $ \polp$ is a convex polytope, $\ext \polp =
  \OO \cap \liet$ and $\ext \polp$ is a unique $W$-orbit.
\end{teo}

There is a unique $K$-invariant complex structure $J$ on $\OO$ such
that $\om$ be a \Keler form. It can be described as follows (see
\cite[p. 113]{huckleberry-DMV} for more information).  Fix a maximal
torus $T$ and a system of positive roots in such that a way $x$
belongs to the closure of the positive Weyl chamber.  Then the complex
structure on $T_x\OO$ is given by the formula
\begin{gather*}
  J u_\alfa =  v_\alfa.
\end{gather*}
% Since
% \begin{gather*}
%   u_\alfa = \frac{ [ x, v_\alfa]} { i\alfa(x)} \qquad v_\alfa = \frac
%   {-[x,u_\alfa]} {i\alfa (x)}
% \end{gather*}
% we compute
% \begin{gather*}
%   \om_x(u_\alfa, v_\alfa) =
%   % \om_x \left ( \frac{ [ x, v_\alfa]} {
%   %     i\alfa(x)} , \frac {-[x,u_\alfa]} {i\alfa (x)} \right )=
%   \left \sx x , \left [ \frac{v_\alfa} { i \alfa (x)} ,\frac{-u_\alfa}
%       { i \alfa (x)} \right ] \right\xs
%   =\frac{ \sx x , [u_\alfa, v_\alfa ] \xs } {(i\alfa(x))^2} = \\
%   = \frac{-B(x, iH_\alfa) } {(i\alfa(x))^2} = \frac{1}{-i\alfa(x) }
% \end{gather*}
% so the sign of $\om_x(u_\alfa, v_\alfa) $ is $\eps_\alfa$. This shows
% that $\om$ is indeed positive with respect to $J$.
Set $G=K^\C$.  The action of $K$ on $\OO$ extends to an action $ G
\times \OO \ra \OO$ which is holomorphic. If $v_\OO$ denotes the
fundamental vector field induced by $v\in \lieg = \liek^\C$, this
implies that
\begin{gather*}
  (iv)_\OO = J v_\OO.
\end{gather*}
Let
\begin{gather*}
  \lieb_- := \liet^\C \oplus \bigoplus_{\alfa \in \roots_+ } \lieg_{-\alfa}
\end{gather*}
denote the negative Borel subalgebra and let $B_-$ be the
corresponding Borel subgroup.   The following lemma is well-known.
\begin{lemma}
  \label{Borel}
  Let $T\subset K$ be a maximal torus and let $\root_+$ be a set of
  positive roots. If $x\in \OO\cap \liet$, then $x\in \cchamber$ if
  and only if  $B_-$ is contained in the
  stabilizer $G_x$.
\end{lemma}
% \begin{proof}
%   ObserveBy definition $x\in \cchamber$ iff $-i\alfa(x) \geq 0$ for
%   any $\alfa \in \root_+$. Since $-i\alfa(x) =0$ for $\alfa \not \in
%   E$, $x\in \cchamber$ iff $\eps_\alfa \equiv 1$.  Next recall that
%   the standard negative Borel subgroup is the connected subgroup $B_-
%   \subset G$ with Lie algebra
%   \begin{gather*}
%     \lieb_- = \liet^\C \oplus \bigoplus_{ \alfa \in \root_+}
%     \lieg_{-\alfa}.
%   \end{gather*}
%   Therefore $B_-\subset G_x$ iff $\xi_X (x) = 0$ for any $X \in
%   \lieb_-$.  Since $T$ stabilizes $x$, this conditions holds for free
%   if $X\in \liet^\C$. So $B_-\subset G_x$ iff $\xi_{X_{-\alfa}} (x)
%   =0$ for any $\alfa \in \root_+$. We need to show that this condition
%   is equivalent to $\eps_\alfa \equiv 1$.  We compute:
%   \begin{gather*}
%     X_{-\alfa} = -\frac{1}{\sqrt{2}} (u_\alfa + i v_\alfa)\\
%     \xi_{X_{-\alfa}} (x) = -\frac{1}{\sqrt{2}} \left ([x, u_\alfa] + J
%       [x, v_\alfa] \right ) = -\frac{i\alfa(x)} {\sqrt{2}} \left
%       (-v_\alfa + J u_\alfa \right )
%     =\\
%     = -\frac{i\alfa(x)} {\sqrt{2}} \left (-1 + \eps_\alfa \right )
%     v_\alfa.
%   \end{gather*}
%   Hence $\xi_{-\alfa} (x) =0$ iff $\eps_\alfa =1$.
% \end{proof}

\section{Group theoretical description of the faces}

\label{group}

In this section we prove that all the faces of a coadjoint orbitope
are coadjoint orbitopes and are exposed. These facts will be used
throughout the rest of the paper.

Let $\OO \subset \liek$ be a coadjoint orbit of $K$.  The
\emph{orbitope} $\c$ is by definition the convex hull of $\OO$.

\begin{lemma} \label{extO=O} $\ext \c = \OO$. Moreover for any face
  $F\subset \c$, $\ext F = F \cap \OO$.
\end{lemma}
\begin{proof}
  This fact is common to all orbitopes \cite
  [Prop. 2.2]{sanyal-sottile-sturmfels-orbitopes}.  By construction
  $\ext \c \subset \OO$. On the other hand $\OO$ lies on a sphere,
  hence all points of $\OO$ are exposed extreme points.  This proves
  the first assertion. The second follows from the first and from
  Lemma \ref {ext-facce}.
\end{proof}

A submanifold $M \subset \R^n$ is called \emph{full} if it is not
contained in any proper affine subspace of $\R^n$.

\begin{lemma}
  \label{full}
  Let $K$ be a compact connected semisimple Lie group and let $\OO
  \subset \liek$ be a coadjoint orbit.  The orbit $\OO$ is full if and
  only if every simple factor of $K$ acts nontrivially on $\OO$.
\end{lemma}
\begin{proof}
  Fix $x \in \OO$.  Let $M$ denote the affine hull of $\OO$ in $\liek$
  and let $V$ be the associated linear subspace, i.e. $M = x + V$.  We
  claim that $M$ contains the origin.  Since $\OO$ is $ K$-invariant,
  so are $M$ and $V$. Hence $V$ is an ideal and $V^\perp$ is an ideal
  as well. Write $x= x_0 +x _1$, with $x_0 \in V$ and $x_1 \in
  V^\perp$.  For any $g\in K$, $g x -x \in V$, $gx_0 - x_0 \in V$ and
  $gx_1 - x_1 \in V^\perp$. So $gx_1 - x_1 \in V\cap V^\perp$,
  i.e. $gx_1 =x_1$. This means that $x_1$ is a fixed point of the
  adjoint action. Since $K$ is semisimple, $x_1=0$, $x\in V$ and $M=V$
  as desired.  Let $K_i$, $i=1, \lds, r$ be the simple factors of
  $K$. Since $V$ is an ideal, $V = \bigoplus_{i\in I} \liek_i$ for
  some subset $I $ of $ \{1, \lds, r\}$.  It is clear that \changed{
    $\liek_j\cap V = \{0\}$} if and only if $[\liek_j, V]=0$ if and
  only if $K_{\changed{j}}$ acts trivially on $\OO$. This proves the
  first statement.
\end{proof}

Let $H$ be a compact connected Lie group (not necessarily semisimple)
and let $\OO \subset \lieh$ be an orbit.  There is a splitting of the
algebra
\begin{gather}
  \label{scomposizione-h}
  \lieh =\liez \oplus \liek_1 \oplus \cds \oplus \liek_r
\end{gather}
where $\liez$ is the center of $\lieh$ and $\liek_i$ are simple
ideals. Let $K_i$ be the closed connected subgroups of $H$ with $\Lie
K_i = \liek_i$. So $H=Z \cdot K_1 \cds K_r$, where $Z$ is the
connected component of the identity in the center of $H$. Any two of
these subgroups have finite intersection. We can reorder the factors
in such a way that $K_i$ acts nontrivially on $\OO$ if and only if $
1\leq i \leq q$ for some $q$ between $1$ and $r$.  Set
\begin{gather*}
  L := K_1 \cds K_q \qquad L' := K_{q+1} \cds K_r .
\end{gather*}
By construction there is a decomposition
\begin{gather}
  \label{eq:H-dec}
  H = Z \cd L \cd L'.
\end{gather}
Any two factors in this decomposition have finite intersection.

\begin{lemma}\label{lemma-z0}
  For any $x\in \OO $, there is a unique decomposition $x=x_0 + x_1$
  with $x_0 \in \liez$ and $x_1 \in \liel$. Moreover
  \begin{gather*}
    \OO = H \cdot x = x _0 + L \cd x_1,
  \end{gather*}
  the affine span of $\OO$ is $x_0 + \liel$ and $x_0$ belongs to
  $\relint \c$.
\end{lemma}
\begin{proof}
  Write $x=x_0 + x_1 + x_2$ with $x_0 \in \liez$, $x_1\in \liel$ and
  $x_2 \in \liel'$. Since $L' \cdot x= x$, the component $x_2$ is
  fixed by $L'$. Since $L'$ is semisimple, this forces $x_2=0$.  It
  follows immediately that $H\cd x = x_0 + L \cd x_1$. By definition
  all simple factors of $L$ act nontrivially on $L\cd x_1$, hence the
  orbit $L\cd x_1$ is full in $\liel$ by Lemma \ref{full}. This proves
  that $\aff (\c) = x_0 + \liel$.  Since $\OO \subset x_0 + \liel$ and
  $\liel \perp x_0$, $x_0$ is the closest point to the origin. Such a
  point is unique because $\c$ is convex. \changed{Since $x_0\in
    \liez$, it is fixed by $H$.}  The last statement follows from
  Lemma \ref{convex-orbit}.
\end{proof}

The statement about the affine span is equivalent to $L\cd x_1$ being
full in $\liel$.  Therefore after possibily replacing $K$ by $L$ and
translating by $x_0$ we can assume for most part of the paper that
$\OO$ is full.

We are interested in the facial structure of $\c$ and we start by
considering the structure of exposed faces.

\begin{lemma}
  \label{sono-simplettico}
  Assume that $K$ is a compact connected Lie group, that $H\subset K$
  is a connected Lie subgroup of maximal rank and that $\OO$ is a
  coadjoint orbit of $K$. Then
  \begin{enumerate}
  \item $\OO \cap \lieh$ is a union of finitely many $H$-orbits;
  \item if $H$ is the centralizer of a torus, then $\OO \cap \lieh$ is
    a symplectic submanifold of $\OO$.
  \end{enumerate}

\end{lemma}
\begin{proof}
  Let $T$ be a maximal torus of $K$ contained in $H$.  Since $\OO \cap
  \lieh$ is an $H$-invariant subset of $H$ and $T$ is a maximal torus
  of $H$ we have $\OO \cap \lieh = H \cd (\OO \cap \liet)$. But
  $\OO\cap \liet$ coincides with ffan orbit of the Weyl group and is
  therefore finite. Hence $\OO\cap \lieh$ is a finite unione of
  $H$-orbits.  This proves the first statement.  For the second assume
  that $H = Z_K(S)$ where $S \subset K$ is a torus. Then $\OO\cap
  \lieh = \OO^S$ is the set of fixed points of $S$, hence it is a
  symplectic submanifold of $\OO$.
\end{proof}

We start the analysis of the face structure of $\c$ by looking at the
esposed faces. At the end of the section we will prove that all faces
are exposed.

Let $u$ be a nonzero vector in $\liek$ and let $\mom_u: \OO \ra \R$ be
the function $\mom_u(x) : = \sx x, u\xs$.  Set
\begin{gather*}
  \ml(\mom_u) : =\{ x \in \OO: \mom_u(x) = \max_\OO \mom_u\}.
\end{gather*}
$\mom_u$ is just the component of the momentum map along $u$.  Then for
$x\in \OO$ and $u,v \in \liek$
\begin{gather*}
  d \mom _u(x) (v_\OO) = \om_x (u_\OO(x), v_\OO(x)) = \sx x, [u,v] \xs
  = \sx [x,u],v\xs.
\end{gather*}
This implies that $x\in \OO$ is a critical point of $\mom_u$ if and
only if $x \in \liez_\liek(u)$, i.e.
\begin{gather*}
  \Crit(\mom_u) = \OO\cap \liez_\liek(u).
\end{gather*}

\begin{lemma}\label{hesso}
  Let $H = Z_K( u)$ be the centraliser of $u$ in $K$ and let $F_u(\c)$
  be the exposed face of $\c$ defined by $u$.  Then
  \begin{enumerate}
  \item $\ml(\mom_u)$ is an $H$-orbit;
  \item $\ext F_u (\c) = \ml(\mom_u)$, so $\ext F_u(\c)$ is an
    $H$-orbit;
    % and $F_u (\c) = \conv (\ml(\mom_u))$;
  \item $F_u (\c) \subset \liez_\liek(u)$.
  \end{enumerate}
\end{lemma}
\begin{proof}
  By Atiyah theorem \cite{atiyah-commuting} the level sets of $\mom_u$
  are connected. In particular $\ml(\mom_u)$ is a connected component
  of $\Crit(\mom_u)$. By the previous lemma it is an $H$-orbit. This
  proves (i).  Let $h_\c $ denote the support function of $\c$, see
  \eqref{def-support}.  Since $\sx \cd , u\xs $ is a linear function,
  its maximum on $\c$, that is $h_\c (u)$, is attained at some extreme
  point, i.e. on $\OO$. Hence
  \begin{gather*}
    \max_{\OO} \mom_u = h_\c (u).
  \end{gather*}
  By Lemma \ref{extO=O} $\ext F_u (\c) = F_u (\c) \cap \OO = \{x\in
  \OO : \sx x, u \xs = h_\c (u) \} = \ml(\mom_u)$. It follows
  immediately that $F_u (\c) = \conv (\ml(\mom_u))$.  Finally (iii)
  follows from the fact that $\ml(\mom_u) \subset \Crit(\mom_u) = \OO
  \cap \liez_\liek(u)$.
\end{proof}

\begin{lemma}\label{hesso-2}
  Fix a maximal torus $T \subset K$, a nonzero vector $u \in \liet$
  and a point $x\in \OO \cap \liet $. Then $x \in \Crit(\mom_u)$ and
  $x$ is a maximum point of $\mom_u$ if and only if there is a Weyl
  chamber in $\liet$ whose closure contains both $x$ and $u$.
\end{lemma}
\begin{proof}
  By assumption $x\in \liet \subset \liez_\liek (u) $ and $
  \liez_\liek (u) = \Crit(\mom_u)$.  To check the second assertion
  recall that $\mom_u$ is a Morse-Bott function with critical points
  of even index (this is Frankel theorem, see e.g. \cite[Thm. 2.3,
  p. 109]{audin-torus-actions} or
  \cite[p. 186]{mcduff-salamon-symplectic}) and any local maximum
  point is an absolute maximum point (see
  e.g. \cite[p. 112]{audin-torus-actions}). Therefore $x$ is a maximum
  point if and only if the Hessian $D^2\mom_u (x)$ is negative
  semidefinite.  Recall that $ T_{x} \OO = \im \ad x$ and that
  \begin{gather*}
    f:=\ad x \restr{T_x\OO} : T_x\OO \ra T_x\OO
  \end{gather*}
  is invertible.  If $w\in T_{x} \OO$, then $w= z_\OO(x)=[z,x]$ for
  some $z\in \liek$. The vector $z$ can be chosen (uniquely) inside
  $T_x\OO$, i.e.  $z = - f\meno (w)$.  Set $ \ga(t) : = \Ad(\exp tz)
  \cdot x$.  Then $\ga(0) = x$, $\dot{\ga}(0) = [z,x]=w$,
  $\ddot{\ga}(0) = [z, [z, x]]$, so
  \begin{gather*}
    D^2\mom_u(x)(w, w) = \dfrac{\mathrm{d^2}}{\mathrm{dt^2}} \bigg
    \vert _{t=0} h(\ga(t)) = \sx \ddot{\ga}(0), u\xs =
    \sx [z,x], [u,z]\xs = \\
    =\sx w, [u,z]\xs = - \sx w, \ad u \circ f \meno (w) \xs.
  \end{gather*}
  (One can prove the same formula much more generally and by a more
  geometric argument, see
  \cite[Prop. 2.5]{heinzner-schwarz-stoetzel}.)  Thus the quadratic
  form $D^2\mom_u (x)$ is negative semidefinite if and only if the
  operator $\ad \, u \circ f\meno$ is positive semidefinite.  This
  operator preserves each $Z_\alfa$ and its restriction to $Z_\alfa$
  is just multiplication by $\alfa(u) / \alfa (x)$. Hence it is
  positive semidefinite iff and only iff $\alfa (u) \alfa(x) \geq 0$
  for any $\alfa \in \root$. This is equivalent to the condition that
  $x$ and $u$ lie in the closure of some Weyl chamber (see e.g. \cite
  [p. 11] {heckman-thesis}).
\end{proof}

The computation above goes back to the work
\cite{duistermaat-kolk-varadarajan} of Duistermaat, Kolk and
Varadarajan and to Heckman's thesis \cite{heckman-thesis}.

\begin{lemma}
  \label{facciona-esposta}
  Let $F = F_u (\c) $ be an exposed face of $ \c$. Set
  $S:=\overline{\exp(\R u) } $ and $H=Z_K(S)$. Then (a) $S$ is a
  nontrivial torus and fixes $F$ pointwise, (b) $\ext F$ is an adjoint
  orbit of $H:= Z_K(S)$, (c) $F \subset \lieh$.
\end{lemma}
\begin{proof}
  Since $u\neq 0$ by the definition of exposed face, $S$ is a
  nontrivial torus. (b) follows from Lemma \ref{hesso}. Moreover $\ext
  F = \ml(\mom_u) \subset \Crit (\mom_u)$. Since $\mom_u$ is the
  Hamiltonian function of the fundamental vector field on $\OO$
  associated to $ u $, $\ext F$ is fixed by $\exp(\R u)$ hence by $S$,
  thus proving (a). Finally $\ext F \subset \liez_\liek(u) = \lieh$ by
  Lemma \ref{hesso}.
\end{proof}

\begin{teo}\label{facciona-orbita}
  Let $F$ be a proper face of $\c$. Then there is a nontrivial torus
  $S \subset K$ with the following properties: (a) $S$ fixes $F$
  pointwise, (b) $\ext F$ is an adjoint orbit of $H:= Z_K(S)$, (c) $F
  \subset \lieh$, (d) $g \cd F =F$ for any $g\in H$.
\end{teo}
\begin{proof}
  (d) is a direct consequence of (c).  To prove (a)-(c) fix a chain of
  faces $F=F_0 \subsetneq F_1 \subsetneq \cds \subsetneq F_k = \c$,
  such that for any $i$ there is no face strictly contained between
  $F_{i-1}$ and $ F_i$.  This is possible by Lemma \ref{face-chain}.
  We will prove (a)-(c) by induction on $k$. If $k=1$, then $F$ is a
  maximal proper face. Since any face is contained in an exposed face,
  $F$ is necessarily exposed.  Thus (a)-(c) follow from the previous
  lemma.  We proceed with the induction.  Let $k>1$ and assume that
  the theorem is proved for faces contained in a maximal chain of
  length $k-1$. Fix $F$ with a maximal chain as above of lenght
  $k$. By the inductive hypothesis the theorem holds for $F_1$, so
  that there is a nontrivial subtorus $S_1 \subset K$ which pointwise
  fixes $F_1$. Moreover if we set $H_1 = Z_K (S_1)$ and $\lieh_1 =
  \Lie H_1 = \liez_\liek (\lies_1)$, then $F_1 \subset \lieh_1$ and
  $\ext F_1$ is an orbit of $H_1$. In particular if we choose a point
  $x\in \ext F \subset \ext F_1$, then $\ext F_1= H_1 \cdot x$.  Split
  $H_1 = Z \cd L \cd L'$ with $Z=Z(H_1)^0$ as in \eqref{eq:H-dec} and
  write $x = x_0 +x _1$ as in Lemma \ref{lemma-z0}, with $x_0 \in
  \liez= \liez(\lieh_1)$ and $x_1 \in \liel$, so that $\ext F_1 = x_ 0
  + L\cd x_1$. The orbit $\c':=L\cd x_1$ is full in $\liel$ and $F':=F
  _ 0 - x_ 0 = F - x_0$ is a maximal face of $\c'$. Therefore $F'$ is
  an exposed face, i.e. there is some $u\in \liel$ such that
  $F'=F_u(\c')$.  Set $S_2 := \overline{\exp ( \R u)}$.  By the
  previous lemma $\ext F'$ is an orbit of $Z_L(S_2)$.  Moreover $x_1
  \in \ext F'$, because $x \in \ext F$. Therefore $\ext F ' =Z_L(S_2)
  \cd x_1$. Since $u\in \liel$ and $\liel \subset \lieh_1$, $u$
  commutes with $\lies_1$. So $S_1$ and $S_2$ commute and generate a
  torus $S$. Set $H:= Z_K(S)$. If $g \in H$, then $g $ commutes with
  $S_1$, hence $g\in H_1$. It follows that $H \subset
  Z_{H_1}(S_2)$. Conversely, if $g\in Z_{H_1}(S_2)$, then $g$ also
  commutes with $S$, so $g \in H$. Thus we get $H=Z_{H_1}(S_2)$.
  Since $S_2 \subset L$, $Z\cd L ' \subset Z_{H_1}(S_2) = H$ and $H =
  Z\cd Z_L(S_2) \cd L'$. Since $Z\cd L'$ fixes $x_1$ this implies that
  $H\cd x_1 = Z_{H_1}(S_2) \cd x_1 = Z_L(S_2) \cd x_1 = \ext F'$.
  Since $x_0 \in \liez =\liez(\lieh_1) $, we conclude that $\ext F =
  \ext F' + x_0 = H \cd x_1 + x_0 = H\cd x$. This proves (b). Next
  observe that the previous lemma also ensures that $F' \subset
  \liez_{\liel}(u)$ and that $\liez_\liel(u) \subset \lieh$. Since
  $x_0\in \lieh$ too, we conclude that $F = F' + x_0 \subset
  \lieh$. This proves (c). By definition $\lieh = \liez_\liek(\lies)$,
  so $S$ fixes any point of $\lieh$ and in particular it fixes
  pointwise $F$. Thus (a) is proved.
\end{proof}

We remark that the inductive argument used in the previous proof does
not imply that all faces are exposed, since being an exposed face is
not a transitive relation.

\begin{cor}\label{F-simplettica}
  If $F\subset \c$ is a face, then $\ext F$ is a symplectic
  submanifold of $\OO$.
\end{cor}
\begin{proof}
  Let $S$ and $H$ be as in Theorem \ref{facciona-orbita}. Then $\ext F
  \subset \OO\cap \lieh$ is an $H$-orbit. The result follows directly
  from Lemma \ref{sono-simplettico}.
\end{proof}

\begin{cor} \label{torus-preserves} If $F\subset \c$ is a face, there
  is a maximal torus $T \subset K$ that preserves $F$.
\end{cor}
\begin{proof}
  A maximal torus of $H$ is also a maximal torus of $K$.
\end{proof}

\begin{remark}
  \label{carat-1}
  The above results shows that every face of $\c$ is a coadjoint
  orbitope for some subgroup $H \subset K$.  One might wonder if a
  similar property holds for all orbitopes: if a group $K$ acts
  linearly on $V$ and $\OO$ is an orbit, one might ask if every face
  of $\c$ is an orbit of some subgroup of $K$. The answer is negative
  in general.  Counterexamples are provided e.g.  by convex envelopes
  of orbits of $S^1$ acting linearly on $\R^n$. These are called
  \emph{Carath\'eodory orbitopes}, since their study goes back to
  \cite{caratheodory-Koeffizienten}. In \cite{smilansky-1985} there is
  a thorough study of the 4-dimensional case (see also
  \cite{barvinok-novik-centrally-symmetric}).  It turns out (see
  Theorem 1 in \cite{smilansky-1985}) that there are many
  1-dimensional faces whose extreme sets are not orbits of any
  subgroup of $K$.  Therefore the fact that we just established,
  namely that the faces of a coadjoint orbitope are all orbitopes of
  the same kind, seems to be a rather remarkable property.
\end{remark}

The subgroups $S$ and $H$ in Theorem \ref{facciona-orbita} are not
unique. Later in Theorem \ref{tutte-esposte} (d) we will show that
there is a canonical choice. Now we wish to show that one can always
assume that $S=Z(H)^0$.

% \begin{lemma}\label{lemma-y00}
%   Let $F$ be a face of $\c$ and $H\subset K$ a connected subgroup of
%   maximal rank such that $\ext F$ is an $H$-orbit. Then $Z(H)^0$
%   acts trivially on $F$ and $F \subset \lieh$.
% \end{lemma}
% \achtung{la prima affermazione serve nel prossimo corollario.  Non
%   mi ricordo dove si usa la seconda affermazione. Boh? Si puo'
%   togliere forse? In tal caso si potrebbe incollare il lemma con il
%   corol}
%
% \begin{proof}
%   Let $p: \liek \ra \lieh$ denote the orthogonal projection.  $H$
%   acts on $\OO$ in a Hamiltonian way with momentum map $p\restr{\OO}
%   $.  If $x \in \ext F$, then $H\cd x = \ext F$ is a symplectic
%   orbit by Corollary \ref{F-simplettica}. Therefore $H_x = H_{p(x)
%   }$, see e.g. \cite[Thm. 26.8,
%   p. 196]{guillemin-sternberg-techniques}. Since $p(x) \in \lieh$,
%   the stabilizer $H_{p(x)}$ contains the center of $H$. So $Z(H)
%   \subset H_x$. This proves the first statement. But from $Z(H) \cd
%   x = x$ it follows also that $x \in \Lie ( Z_K(Z(H)))$. Since $H$
%   is connected and has maximal rank, $Z_K(Z(H) )^0
%   =H$. % Podesta p. 26
%   So $x \in \lieh$.
% \end{proof}
%

\begin{cor} \label{H-buoni} In Theorem \ref{facciona-orbita} we can
  assume that $Z(H)$ acts trivially on $F$ and that $S=Z(H)^0$.
\end{cor}
\begin{proof}
  Let $p: \liek \ra \lieh$ denote the orthogonal projection.  $H$ acts
  on $\OO$ in a Hamiltonian way with momentum map $p\restr{\OO} $.  If
  $x \in \ext F$, then $H\cd x = \ext F$ is a symplectic orbit by
  Corollary \ref{F-simplettica}. Therefore $H_x = H_{p(x) }$, see
  e.g. \cite[Thm. 26.8, p. 196]{guillemin-sternberg-techniques}. Since
  $p(x) \in \lieh$, the stabilizer $H_{p(x)}$ contains the center of
  $H$. So $Z(H) \subset H_x$. This proves the first statement.  Next
  set $S' = Z(H)^0$.  Then $S'$ is a positive dimensional torus.  To
  prove the second fact it is enough to show that changing $S$ to $S'$
  does not change the centralizer, i.e. that $H=Z_K(S')$.  Since $S'
  \subset Z(H)$, $H$ and $S'$ commute, so $H \subset Z_K(S')$.  On the
  other hand $H$ is the centralizer of $S$, so $S \subset S'$, and
  $Z_K(S') \subset Z_K(S)=H$. Therefore indeed $H=Z_K(S')$.
\end{proof}

% Let now $F$ be a proper face and $H$ be any connected subgroup of
% $K$ such that $\ext F$ is an $H$-orbit. (By Theorem
% \ref{facciona-orbita} such a subgroup always exists.  In general
% there are various choices for $H$. Later in \eqref{def-hf} we will
% show that there is a maximal one.)  Split the algebra
% \begin{gather}
%   \label{scomposizione-h}
%   \lieh =\liez \oplus \liek_1 \oplus \cds \oplus \liek_r
% \end{gather}
% where $\liez$ is the center of $\lieh$ and $\liek_i$ are simple
% ideals. Let $K_i$ be the closed connected subgroups of $H$ with
% $\Lie K_i = \liek_i$. So $Z=Z(H)^0$, $H=Z \cdot K_1 \cds K_r$ and
% any two of these subgroups have finite intersection. We can reorder
% the factors in such a way that $K_i$ acts nontrivially on $F$
% (equivalently on $\ext F$) if and only if $ 1\leq i \leq q$ for some
% $q$ between $1$ and $r$.  Set
% \begin{gather*}
%   L := K_1 \cds K_q \qquad L' := K_{q+1} \cds K_r .
% \end{gather*}
% By construction there is a decomposition
% \begin{gather}
%   \label{eq:H-dec}
%   H = Z \cd L \cd L'.
% \end{gather}
% Again any two factors in this decomposition have finite
% intersection.
%
%
The following is an immediate consequence of Lemma \ref{lemma-z0}.
\begin{lemma}\label{lemma-y0}
  Let $F$ be a face of $\c$, $H\subset K$ a connected subgroup and
  assume that $\ext F$ is an $H$-orbit and that $F \subset \lieh$.
  Decompose $H$ as in \eqref{eq:H-dec}, i.e. $L$ is the product of the
  simple factors of $(H,H)$ that act nontrivially on $F$, while $L'$
  is the product of those factors that act trivially. If $x\in \ext F
  $, then $x=x_0 + x_1$ with $x_0 \in \liez$ and $x_1 \in
  \liel$. Moreover
  \begin{gather*}
    \ext F = H \cdot x = x _0 + L \cd x_1
  \end{gather*}
  and $L \cdot x_1 \subset \liel$ is full.
\end{lemma}
% \begin{proof}
%   Write $x=x_0 + x_1 + x_2$ with $x_0 \in \liez$, $x_1\in \liel$ and
%   $x_2 \in \liel'$. Since $L' \cdot x= x$, the component $x_2$ is
%   fixed by $L'$. Since $L'$ is semisimple, this forces $x_2=0$.  It
%   follows immediately that $H\cd x = x_0 + L \cd x_1$. By definition
%   all simple factors of $L$ act nontrivially on $L\cd x_1$, which
%   means that it is a full orbit.
% \end{proof}
%
%
%

% If $T\subset K$ is a maximal torus, we denote by $W(K,T)$ or simply
% by $W$ the Weyl group of $(K,T)$. We let $\pi: \liek \ra \liet$
% denote the orthogonal projection with respect to the scalar product
% $\sx \ , \ \xs= - B$. Its restriction to $\OO$ is denoted by $\momt
% : \OO \ra \liet$; it is the momentum map for the $T$-action on
% $\OO$.  $\polp: = \momt(\OO)$ is the momentum polytope.  The
% following convexity theorem of Kostant \cite{kostant-convexity} is
% the basic ingredient in the whole theory.
% \begin{teo}[Kostant] \label{Kostant} Let $K$ be a compact connected
%   Lie group, let $T \subset K$ be a maximal torus and let $\OO$ be a
%   coadjoint orbit. Then $ \polp$ is a convex polytope, $\ext \polp =
%   \OO \cap \liet$ and $\ext \polp$ is a unique $W$-orbit.
% \end{teo}

Now we fix a maximal torus $T$ and we use the notation of
p. \pageref{momentum-notation-page}. We wish to show that the
$T$-stable faces of $\c$ and the faces of the momentum polytope are in
bijective correspondence. This will be used to prove that all faces of
$\c$ are exposed. The relation between the $T$-invariant faces of $\c$
and the faces of $P$ will be studied further in the next section.

The following lemma is a consequence of Kostant convexity theorem.
See \cite[Lemma 7]{gichev-polar} for a proof in the context of polar
representations.

\begin{lemma}
  \label{gichev}
  Let $K$ be a compact connected Lie group, $T \subset K$ be a maximal
  torus and let $\pi: \liek \ra \liet$ be the orthogonal projection.
  Then (i) If $E \subset \liek$ is a $K$-invariant convex subset, then
  $E\cap \liet = \pi (E)$. (ii) If $A \subset \liet$ is a
  $\Weyl$-invariant convex subset, then $K \cdot A $ is convex and
  $\pi (K\cdot A )= A$.
\end{lemma}

\begin{lemma} \label{proiezione-intersezione} Let $T \subset K$ be a
  maximal torus and let $F \subset \c$ be a nonempty $T$-invariant
  face. Set $ \sigma : = \pi(\ext F)$.  Then
  % $\ext \sigma = \ext F \cap \liet$ and
  $\sigma = \pi(F) = F\cap \liet$.  Moreover $\sigma$ is a nonempty
  face of the momentum polytope $P$.
\end{lemma}
\begin{proof}
  We prove this lemma in the same way as Kostant theorem is deduced
  from the Atiyah-Guillemin-Sternberg theorem.  By Corollary
  \ref{F-simplettica} $\ext F$ is a symplectic submanifold of
  $\OO$. $T$ acts on $\ext F$ with momentum map given by the
  restriction of $\pi $ to $\ext F$.  By definition $\sigma=\pi(\ext
  F)$ is the momentum polytope for this action.  By the
  Atiyah-Guillemin-Sternberg theorem
  \begin{gather*}
    \sigma = \conv \pi \left ( \left (\ext F \right )^T \right ) =
    \conv \pi ( \ext F \cap \liet).
  \end{gather*}
  This means first of all that $\sigma $ is convex. Since $\pi$ is
  linear it follows that $\pi(F) = \conv \pi (\ext F) = \sigma$.  On
  the other hand, since $\pi (\ext F\cap \liet)= \ext F\cap \liet$, we
  get
  \begin{gather}
    \label{estremi-rompipalle}
    \ext \sigma \subset \ext F \cap \liet \qquad \sigma \subset F\cap
    \liet .
    % \sigma = \conv \pi (\ext F \cap \liet) \subset \conv \pi (F\cap
    % \liet) = \conv F\cap \liet = F\cap \liet .
  \end{gather}
  % $\pi(\ext F) = \conv \pi (\ext F \cap \liet) \subset F\cap \liet$
  % and also
  Conversely $F\cap \liet = \pi(F\cap \liet) \subset \pi (F)$.  Since
  $\pi(F) = \sigma$ we get indeed $F\cap \liet = \sigma$.  Thus the
  first part is proven.  In particular we can apply this with $F=\c$,
  and we get that $P= \c \cap \liet$.  That $F\cap \liet$ is a face of
  $P$ now follows directly from Lemma \ref{micro-convesso} without
  assuming that $F$ be $T$-invariant.  To check that $\sigma \neq
  \vacuo$, recall that if a torus acts on a compact \Keler
  \changed{manifold} in a Hamiltonian way, then it has some fixed
  points. So $(\ext F) ^T = \ext F \cap \liet \neq \vacuo$ and $\sigma
  \neq \vacuo$.
\end{proof}

Recall the following basic property of Hamiltonian actions (see
e.g. \cite[Thm. 3.6]{guillemin-sternberg-convexity-1}).
\begin{lemma}
  \label{minimo-simplettico}
  Let $M$ be a symplectic manifold and let $T$ be a torus that acts on
  $M$ in a Hamiltonian way with momentum map $\mom : M \ra \liet$. If
  $S \subset T $ is a subtorus that acts trivially on $M$, then
  $\mom(M)$ is contained in a translate of $\lies^\perp$.
\end{lemma}

If $M\subset \R^n$ is an affine subspace, the linear subspace parallel
to $M$ is called the \emph{direction} of $M$
\cite[p. 42]{berger-geometry-I}. Denote by $\sigma^\perp$ the
orthogonal space in $\liet$ to the direction of $\sigma $.

\begin{lemma}
  \label{medio-simplettico}
  Let $F$ be a proper face of $\c$, let $H$ be a subgroup as in
  Theorem \ref{facciona-orbita} and let $T $ be a maximal torus of
  $H$. Then $\sigma := F\cap \liet$ is a proper face of $P$ and $ \ext
  F $ is a $ Z_K(\sigma^\perp) $-orbit.
\end{lemma}
\begin{proof}
  By assumption $\ext F$ is an $H$-orbit. Hence it is a connected
  component of $\OO\cap \lieh$. In particular by Lemma
  \ref{sono-simplettico} it is a symplectic submanifold of $\OO$.  By
  Corollary \ref{H-buoni} $S: = Z(H)^0$ is a nontrivial subtorus of
  $T$, which acts trivially on $\ext F$.  The momentum map for the
  $T$-action is the restriction of $\pi$.  So by Lemma
  \ref{minimo-simplettico} $\sigma = \pi(\ext F)$ is contained in a
  traslate of $\lies^\perp$, i.e.  $\lies \subset \sigma^\perp$. It
  follows that $Z_K(\sigma^\perp) \subset Z_K(\lies) = Z_K(S) = H$.
  Next consider the decomposition \eqref {eq:H-dec}. We know that
  $L\cd x_1 \subset \liel$ is a full orbit.  Denoting by $\aff ( \cd
  )$ the affine span
  \begin{gather*}
    \aff F = \aff (\ext F) = x_0 + \liel.
  \end{gather*}
  Since $x_0\in \liet$, $(x_0 + \liel) \cap \liet = x_0 + (\liel \cap
  \liet)$ and
  \begin{gather*}
    \aff \sigma = \aff (F\cap \liet) \subset ( \aff F) \cap \liet =
    x_0 + (\liel \cap \liet).
  \end{gather*}
  Since $\liel$ is an ideal of $\liek$, it is the direct orthogonal
  sum of $\liel \cap \liet$ and some $Z_\alfa$, see
  \eqref{Zalfa}. Hence $ \liel= (\liel \cap \liet ) \stackrel{\perp}{
    \oplus} (\liel \cap \liet^\perp)$.  It follows that $\pi(\liel) =
  \liel \cap \liet$ and also, since $x_0 \in \liet$, that $\pi(x_0
  +\liel) = x_0 + (\liel \cap\liet)$. So
  \begin{gather*}
    x_0 + (\liel \cap \liet ) = \pi(x_0 + \liel)= \pi(\aff F) \subset
    \aff (\pi(F)) = \aff \sigma.
  \end{gather*}
  From these two inclusions we get that $\aff \sigma = x_0 + (\liel
  \cap \liet)$. Therefore $\sigma^\perp $ is the orthogonal complement
  of $\liel \cap \liet$ in $\liet$. Since $\liet = \liez
  \stackrel{\perp}{ \oplus} (\liel \cap \liet) \stackrel{\perp}{
    \oplus} (\liel' \cap \liet)$, we get $\sigma^ \perp = \liez \oplus
  (\liel'\cap \liet) \subset \liez \oplus \liel'$.  So $[\liel,
  \sigma^\perp] \subset [\liel, \liez\oplus \liel']=0$ and $L \subset
  Z_K(\sigma^\perp)$. From the inclusions $L\subset Z_K(\sigma^\perp)
  \subset H$ and the fact that $L\cd x = H\cd x = \ext F$ for any
  $x\in \ext F$ we immediately get $Z_K(\sigma^\perp) \cd x = \ext F$.
  We already know (from Lemma \ref{proiezione-intersezione}) that
  $\sigma $ is a nonempty face of $P$.  By Theorem
  \ref{facciona-orbita} $\liez \neq \{0\}$, so $\sigma^\perp \neq
  \{0\}$, $\aff \sigma \neq \liet$ and $\sigma \subsetneq P$. This
  shows that $\sigma$ is a proper face.
\end{proof}

\begin{cor}
  \label{interseco-monotono}
  Let $F_1, F_2$ be a proper faces of $\c$, let $H_1, H_2$ be
  corresponding subgroups as in Theorem \ref{facciona-orbita} and let
  $T $ be a maximal torus of $K$ which is contained in both $H_1 $ and
  $H_2$.  If $F_1 \cap \liet = F_2 \cap \liet$, then $F_1 = F_2$.
\end{cor}
\begin{proof}
  Set $\sigma :=F_i \cap \liet$.  Recall from
  \eqref{estremi-rompipalle} that $\ext \sigma \subset \ext F_i$ and
  pick $x\in \ext \sigma$. Then we can apply the previous lemma to
  both faces and we get $\ext F_1 = Z_K(\sigma^\perp) \cd x = \ext
  F_2$. The result follows.
\end{proof}

If $F \subset \c $ is a face set
\begin{gather}
  \label{def-hf}
  \begin{gathered}
    H_F :=\{ g\in K: gF=F\} \qquad Z_F := Z(H_F)^0
    \\
    \CF : = \{ u \in \liek: F=F_u (\c)\}.
  \end{gathered}
\end{gather}

The following is the main result of this section.
\begin{teo}
  \label{tutte-esposte}
  All proper faces of $\c$ are exposed. More precisely, if $F$ is a
  proper face $F \subset \c$, then
  \begin{enumerate}
  \item if $T\subset H_F$ is a maximal torus, $u\in \liet$ and $ F\cap
    \liet = F_u(P) $, then $F=F_u(\c)$;
  \item there is a vector $u\in \liez_F$ such that $F=F_u(\c)$;
  \item if $u \in \CF \cap \liez_F$, then $H_F = Z_K( u)$ (in
    particular $H_F$ is connected and $Z_F$ has positive dimension);
  \item the subgroup $H_F$ satisfies (a)-(d) of Theorem
    \ref{facciona-orbita}.
  \end{enumerate}
\end{teo}
\begin{proof}
  We start by proving (a) under the assumption that the maximal torus
  $T$ is contained in some subgroup $H$ that has the properties listed
  in Theorem \ref{facciona-orbita}.  By Lemma \ref{medio-simplettico}
  $\sigma : = F \cap \liet = F\cap \polp$ is a proper face of $\polp$.
  Since all faces of a polytope are exposed
  \cite[p. 95]{schneider-convex-bodies}, there is a vector $u\in
  \liet$ such that $\sigma$ equals the exposed face of $\polp$ defined
  by $u$, i.e.  $ \sigma = F_u (\polp)$.  Since $u \in \liet$ and
  $P=\pi(\OO)$, $h_\polp(u) =\max_{x\in \OO} \sx u, x\xs = h_\c (u)$.
  Set $F' : = F_u(\c)$.  $F'$ is a $T$-invariant face since $u$ is
  fixed by $T$. We wish to show that $F=F'$. The inclusion $F \subset
  F'$ is immediate.  Indeed if $x \in F$, then $\pi(x) \in \sigma$, so
  $\sx x, u \xs = h_P(u)= h_\c(u)$. It is also immediate that $F'\cap
  \liet = \sigma $.  So we have two faces $F$ and $F'$ with $F\cap
  \liet = F'\cap \liet = \sigma$.  Set $H':=Z_K(u)$. By Lemma
  \ref{hesso} $\ext F' = \ml(\mom_u)$ is an $ H'$-orbit and $H'$
  satisfies (a)-(d) of Theorem \ref{facciona-orbita} for $F'$.
  Clearly $T\subset H'$ since $u\in \liet$, and by hypothesis also
  $T\subset H$.  We can therefore apply Corollary
  \ref{interseco-monotono} and we get $F=F'$. In particular $F= F_u
  (\c)$ is an exposed face.  We have thus proved (a) under the
  assumption that $T \subset H$ for some $H$ as in Theorem
  \ref{facciona-orbita}.  Next we show that the vector $u$ can be
  chosen inside $\liez_F$.  The subgroup $H_F \subset K$ is compact
  and preserves both $\c$ and $F$.  By Proposition \ref {u-cono} there
  is a vector $u\in \CF$ that is fixed by $H_F$.  Note that $H_F$ is
  of maximal rank since $H\subset H_F$.  If $T$ is a maximal torus
  contained in $H_F$, then $u$ is is fixed by $T$, so $u \in \liet
  \subset \lieh_F$. It follows that $u\in \lieh_F$ and since $H_F$
  fixes $u$ it follows that $u\in \liez_F$. Thus (b) is proved.  To
  prove (c) assume that $u\in \liez_F$ and that $F = F_u (\c)$.  Then
  $H_F \subset Z_K(u)$ since $u\in \liez_F$.  On the other hand $\ext
  F = F_u(\c) \cap \OO = \ml(\mom_u) = Z_K(u) \cd x$ by Lemma
  \ref{hesso}. Therefore $Z_K(u)$ preserves $F$ and therefore $Z_K(u)
  \subset H_F$ by definition. So $H_F =Z_K(u)$ and (c) is proved.  (d)
  follows from Lemma \ref {facciona-esposta} and the fact that
  $H_F=Z_K(u )$. Now we know that $H_F$ itself has the properties of
  Theorem \ref{facciona-orbita}. Hence (a) holds for any torus $T
  \subset H_F$.
\end{proof}

\begin{remark}
  In general the faces of an orbitope are not necessarily exposed. For
  example 4-dimensional Carath\'eodory orbitopes have non-exposed
  faces, see \cite[Thm. 1(5b)]{smilansky-1985} (note that the author
  uses the word ``facelet'' for face and ``face'' for exposed
  face). It is important to understand whether an orbitope has only
  exposed faces.  Indeed this is Question 1 in
  \cite{sanyal-sottile-sturmfels-orbitopes}. The previous theorem
  shows that this is always the case for coadjoint orbitopes.
\end{remark}

  \begin{cor}
    If $\OO ' \subset \OO$ is a smooth submanifold, then $\conv
    (\OO')$ is a face of $\c$ if and only if there is a vector $u$
    such that $\OO'=\ml (\mom_u)$.
  \end{cor}
  \begin{proof}
    Set $F = \conv (\OO')$.  From the fact that $\OO$ is contained in
    a sphere, it follows as in Lemma \ref{extO=O} that $\ext F =
    \OO'$. Therefore the statement follows immediately from Lemma
    \ref{hesso} and the fact that every face of $\c$ is exposed.
  \end{proof}
  This is a first characterization of the submanifolds that appear as
  $\ext F$ for some face $F$.  In \S \ref {complex-structure} we will
  see that this characterization becomes much more transparent using
  the complex structure of $\OO$.  An explicit characterization in
  terms of root data will be given in \S \ref{satake-section}.

  Various results about the faces have been established using
  \emph{some} subgroup $H$ satisfying the properties stated in Theorem
  \ref{facciona-orbita}.  Now we know that $H_F$ does satisfy these
  properties. Hence we can state those results more cleanly.  This is
  done in Theorem \ref{enunciatone} below. Next in Lemma
  \ref{inclusioni} we will make precise the possible freedom in the
  choice of the group $H$. First of all decompose $H_F$ as in Lemma
  \ref{lemma-y0}:
  \begin{gather}
    \label{eq:HF-dec}
    H_F = Z_F \cd K_F \cd K_F'.
  \end{gather}
  $Z_F$ is defined in \eqref{def-hf}, $K_F$ is the product of the
  simple factors of $(H_F, H_F)$ that act nontrivially on $\ext F$ and
  $K'_F$ is the product of the remaining factors.

  \begin{teo}
    \label{enunciatone}
    Let $T\subset K$ be a maximal torus.
    \begin{enumerate}
    \item If $F \subset \c$ is a proper $T$-invariant face, then $
      \sigma :=F\cap \liet = \pi(F) = \pi(\ext F)$ is a proper face of
      the momentum polytope $P$ and $ \ext F $ is a $
      Z_K(\sigma^\perp) $-orbit.
    \item If $F_1$ and $F_2$ are $T$-invariant proper faces, then $F_1
      \subset F_2$ if and only if $F_1 \cap \liet \subset F_2 \cap
      \liet$.
    \item \label{enunciatone-c} If $F_1$ and $F_2$ are $T$-invariant
      proper faces, then $F_1 = F_2$ if and only if $F_1 \cap \liet =
      F_2 \cap \liet$.
    \item
      \label{lemma-x0-x1}
      If $x\in \ext F$, then $x=x_0 + x_1$ with $x_0 \in \liez_F$ and
      $x_1 \in \liek'_F$. Moreover
      \begin{gather}
        \label{extfx0}
        \ext F = x _0 + K_F \cd x_1
      \end{gather}
      and $K_F \cdot x_1 \subset \liek_F$ is full.
      % \item Given a proper $T$-invariant face $F$, there exists a
      %   maximal $T$-invariant proper face $F'$ containing $F$.
      %   \achtung{Serve?}
    \end{enumerate}
  \end{teo}
  \begin{proof}
    If $F$ is $T$-stable, then $T\subset H_F$. So (a) follows from
    Lemma \ref{medio-simplettico}.  (b) Set $\sigma _i :=F_i \cap
    \liet_i$.  If $F_1 \subset F_2$, then clearly $\sigma _1 \subset
    \sigma_2$. To prove the converse, assume that $\sigma _1 \subset
    \sigma_2$ and pick $x \in \sigma_1$.  Then $Z_K(\sigma^\perp_1)
    \subset Z_K(\sigma_2^\perp)$ and $\ext F_i = Z_K(\sigma_i^\perp)
    \cd x$. Thus $\ext F_1 \subset \ext F_2$.  (c) follows
    immediately.  (d) is just Lemma \ref{lemma-y0} stated for $H=H_F$.
    % (e) $F\cap \liet$ is a proper face of $\polp$; therefore there
    % exists a maximal face $\sigma'$ of $\polp$ such that $\sigma
    % \subset \sigma'$.  Let $u\in \liet$ be such that $\sigma'$
    % equals the exposed face of $\polp$ for $u$. Then $F'=F_u(\OO)$
    % is a $T$-invariant face of $\c$ and it follows from the
    % monotonicity property (b) that it is maximal.
  \end{proof}

  \begin{lemma} \label{inclusioni} If $F\subset \c$ is a face and $H
    \subset K $ is a connected subgroup, such that $ F \subset \lieh $
    and $\ext F$ is an $H$-orbit, then $K_F \subset H \subset H_F$ and
    $K_F =L$.
  \end{lemma}
  \begin{proof}
    Necessarily $F\neq \vacuo$. Since $\ext F$ is an $H$-orbit, $H$
    preserves $\ext F$, hence $F$. So $H \subset H_F$ by definition
    \eqref{def-hf}. To prove the opposite inclusion, split as usual
    $H=Z\cd L \cd L'$ and write $x=x_0 + x_1$ as in Lemma
    \ref{lemma-y0}. The orbit $L\cd x_1 \subset \liel$ is full, so the
    affine span of $F$ is $x_0 + \liel$.  Since also $H_F$ has the
    properties stated in Theorem \ref{facciona-orbita} we can repeat
    the same reasoning for $H_F$ instead of $H$. Thus we get that the
    affine span of $F$ is $x_0 + \liek_F$. Therefore $\liel=
    \liek_F$. So $L$ and $K_F$ are connected subgroups of $K$ with the
    same Lie algebra and therefore coincide. This implies $K_F = L
    \subset H$.
  \end{proof}
  \begin{Example}\label{exe1}
    Set $\liek=\mathfrak{su}(n+1)=\{ X \in \mathfrak{gl} (n+1,\C ) :\,
    X+X^*=0, \ \mathrm{Tr}(X)=0 \} $, $\mathcal{H}=\{X \in \gl (n+1):
    X=X^*\}$ and $\mathcal{H}_1=\{X \in \mathcal{H}:{Tr}(X)=1\}$.  We
    identify $\mathfrak{su}(n+1)$ with $\mathcal{H}_1$ using the map
    \[
    \phi :\su(n+1) \ra \mathcal{H}_1\qquad \phi(X)= iX +
    \frac{\mathrm{Id}_{n+1}}{n+1}.
    \]
    The vector space of Hermitian matrices is endowed with an
    invariant scalar product, given by $\langle A,B \rangle
    =\mathrm{Tr}(AB)$.  Let $\OO \subset \su(n+1)$ be the coadjoint
    orbit corresponding to $\PP^n(\C)$ endowed with the Fubini-Study
    metric. Then $\OO'=\phi(\OO)$ is the set of orthogonal projectors
    onto lines, i.e.
    \[
    \OO'=\{A \in \mathcal{H}:\, A^2=A, \mathrm{rank}(A)=1 \}.
    \]
    Using the spectral theorem it is easy to check that
    \[
    \OO'=\{A \in \mathcal{H}_1:\, A\geq 0,\ \mathrm{rank}(A)=1 \}
    \]
    and
    \[
    \widehat{\OO'}=\{A\in \mathcal{H}_1:\, A\geq 0 \}.
    \]
    Given a Hermitian matrix $u\neq 0$ we wish to study the face
    \begin{gather*}
      F:=F_u(\cp).
    \end{gather*}
    We can assume that $u$ be tangent to $\mathcal{H}_1$, i.e. $\Tr u
    = 0$.  Let
    \[
    \C^{n+1}=V_1 \oplus \cdots \oplus V_s
    \]
    be its eigenspace decomposition, i.e.  $u_{|_{V_i}}=\mu_i
    \mathrm{Id}_{{V_i}}$.  Since $u\neq 0$ and $\Tr u =0 $ $s>1$.  We
    assume $\mu_1 < \mu_2 < \cdots < \mu_s$. Let
  $$
  \mom: \OO' \ra \R \qquad \Phi_u (x)=\langle u,x \rangle$$ be the
  height function with respect to $u$. The critical set of $\mom_u$ is
  $ \{A \in \OO': [A,u]= 0\}$. Since $[A, u]=0$ if and only if $A(V_i)
  \subset V_i$, it follows that this is the set of projectors onto
  lines that are contained in some of the $V_i$'s, i.e.
  $\Crit(\mom_u) = \PP(V_1) \sqcup \cds \sqcup \PP(V_s)$.  For the
  same reason
  $$
  Z_{\SU(n+1)} (u)=\mathrm{S}(\mathrm{U} (V_1) \times \cdots \times
  \mathrm{U} (V_s)).
  $$
  Let $v_i$ be a non zero vector of $V_i$ and let $P_{v_i}$ denote the
  orthogonal projection onto the complex line $\C v_i$. Then
  \[
  % \OO \cap \{ X \in \mathcal{H}:\, [X,u]= 0 \} =\bigsqcup^s_{i=1}
  \PP(V_i) = Z_{\SU(n+1)} (u) \cdot {P}_{v_i}.
  \]
  If $A\in \Crit(\mom_u)$, then
  \[
  \Phi_u(A)=\mu_1 \mathrm{Tr}(A\restr{V_1})+\cdots+ \mu_s
  \mathrm{Tr}(A\restr{V_s}).
  \]
  Since $ \operatorname{Tr} (A\restr {V_i}) \geq 0$ and
  \begin{gather*}
    \sum_{i=1}^s \operatorname{Tr} (A\restr {V_i}) = \operatorname{Tr}
    A = 1
  \end{gather*}
  the maximum of $\Phi_u$ is equal to $\mu_s$ and it is attained
  exaclty on $\PP(V_s)$.  This means that
  \begin{gather*}
    \ext F = \ml(\mom_u) = \PP(V_s) \subset \OO'
    \\
    F = \conv\left ( \PP \left(V_s \right )\right) = \{A \in \mathcal
    {H}_1 : A \geq 0, \, A\restr {V_s^\perp} \equiv 0\}.
  \end{gather*}
  So $F$ consists of the operators in $\cp$ that are supported on
  $V_s$.  Notice that $ H_{F}=\mathrm{S}(\mathrm{U} (V_s)\times
  \mathrm{U} (V_s^\perp )) $ and $ \liez_F = i\R v$ where $v $ is
  the Hermitian operator such that
  \begin{gather*}
    v\restr{V_s}=\frac{\mathrm{Id}}{\dim V_s} \qquad
    v\restr{V_s^\perp} = -\frac{\mathrm{Id} }{\dim V_s^\perp}.
  \end{gather*}
  In fact $F=F_v(\cp)$.  In particular in this example $\CF $ is much
  larger than $\liez_F \cap \CF$.  The above computation shows that to
  each face corresponds a subspace, namely $V_s$.  Viceversa, given a
  subspace $W\subset \C^{n+1}$, let $w$ be the Hermitian operator such
  that
  \begin{gather*}
    w\restr{W}=\frac{\mathrm{Id}}{\dim W} \qquad w\restr{W^\perp} =
    -\frac{\mathrm{Id} }{\dim W^\perp}.
  \end{gather*}
  Then
  \[
  F_w (\c)= \{A\in \mathcal{H}_1:\, A\geq 0, \ %A(W) \subset W \
  % \mathrm{and}\
  A\restr{W^\perp} = 0 \}= \conv (\PP(W)).
  \]
  Therefore the faces of $\cp$ are in one-to-one correspondence with
  the subspaces of $\C^{n+1}$.
\end{Example}

\section{The role of the momentum polytope}
\label{relazione-politopo}

In this section we prove Theorem \ref{main}. We will start by
constructing the inverse of the map considered in Theorem \ref{main}
and we will prove in detail that it passes to the quotient. At the end
(Theorem \ref{prova-main}) we will show that the two maps are inverse
to each other.

Consider a full orbit $\OO \subset \liek$, a maximal torus $T \subset
K$ and the momentum polytope $P$. In this section we will study in
detail the relation between the faces of $\c$ and those of $P$.
Denote by $\faces$ the set of proper faces of $\OO$ and by $\facesp$
the proper faces of the polytope $P$.  If $F$ is a face of $\OO$ and
$a\in K$, then $a\cdot F$ is still a face, so $K$ acts on $\faces$.
Similarly $W=W(K,T)$ acts on $\facesp$.  We wish to show that $\faces
/ K \cong \facesp /W$.

\begin{lemma} \label{facce-1} If $F $ is a face of $\OO$, there is a
  $T$-stable face $F' $ which is conjugate to $F$, i.e. $F' = a\cd F$
  for some $a\in K$.  $F'$ is unique up to conjugation by elements of
  $N_K(T)$.
\end{lemma}
\begin{proof}
  By Corollary \ref{torus-preserves} $F$ is preserved by some maximal
  torus $S \subset K$. There is $a\in K$ such that $S = a\meno
  Ta$. Hence $F' = a\cd F$ is preserved by $T$.  To prove
  uniqueness assume that $F_1$ and $F_2$ be $T$-stable faces of $\OO$
  and that $F_2 = a \cdot F_1$ for some $a\in K$. Then $H_{F_2} = a
  H_{F_1} a\meno$. In particular both $T$ and $ aTa\meno$ are
  contained in $H_{F_2}$, so there is $b\in H_{F_2}$, such that $aT
  a\meno = b T b\meno$. Then $w = b\meno a \in N_K(T)$ and $w \cdot F_1 =
  b\meno a \cdot F_1 = b\meno F_2 = F_2$.
\end{proof}

Define a map
\begin{gather*}
%  \label{mappa-facce-1}
  \phi : \faces / K \ra \facesp /\Weyl
\end{gather*}
by the following rule: given $[F] \in \faces$ choose a $T$-invariant
representative $F$ and set $\phi( [F]): = [F\cap \liet]$.  By Lemma
\ref{proiezione-intersezione} $F\cap \liet$ is indeed a face of the
polytope.  By Lemma \ref{facce-1} if $F'$ is $T$-stable and $[F'] =
[F]$ then $F'\cap \liet$ and $F\cap \liet$ are interchanged by some
element of $\Weyl$. This shows that the map $\phi$ is well-defined.

Now fix a face $F$ of $\c$ and a maximal torus $T \subset H_F$.  Since
$T\cap K_F$ is a maximal torus of $K_F$ and $T\cap K_F'$ is a maximal
torus of $K_F'$, corresponding to the decomposition \eqref{eq:HF-dec}
there is a splitting
\begin{gather*}%\label{eq:dec-hf}
  \liet = \liez_F \oplus (\liet \cap \liek_F) \oplus (\liet \cap
  \liek'_F).
\end{gather*}
Denote by $ W_F$ and $W_F'$ the Weyl groups of $(K_F , K_F \cap T)$
and $(K_F', K_F' \cap T)$ respectively.  $W_F$ and $W_F'$ can be
considered as subgroups of $W$. They commute and have the following
sets of invariant vectors:
\begin{gather*}
  %\label{WF-invarianti}
\liet^{W_F} = \liez_F \oplus
\liek'_F \qquad \liet^{W_F'} = \liez_F \oplus \liek_F
\qquad
  \liet^{W_F \times W_F'} = \liez_F .
\end{gather*}

\begin{lemma} \label{facce-inv} Let $T \subset K$ be a maximal torus
  and let $F$ be a nonempty $T$-invariant face of $\OO$.  Set $ \sigma
  :=F \cap \liet$.   Then  (i) $W_F \times W_F'$ preserves $\sigma$; (ii) $ F
  = H_F \cd \sigma = K_F \cd \sigma$.
\end{lemma}
\begin{proof}
  Recall that $\ext F = x_0 + K_F \cd x_1$. By Kostant theorem $\sigma
  =\pi(\ext F) = \pi(x_0 + K_F\cd x_1) = x_0 + \conv (W_F\cd x_1) =
  \conv (W_F\cd x)$. Hence  $W_F$ preserves $\sigma$. Moreover $\sigma
  \subset \liez_F \oplus ( \liet\cap \liek_F) $ hence $W'_F$ fixes
  $\sigma$ pointwise and (i) follows.  Similarly, since $\sigma
  \subset \liez_F \oplus \liek_F$, $Z_F\cd K'_F$ fixes $\sigma$
  pointwise.  Therefore $H_F \cd \sigma = K_F \cd \sigma$.  By Lemma
  \ref{gichev} $K_F \cd (\sigma -x_0)$ is convex and the same is true
  of $x_ 0 + K_F \cd (\sigma -x_0) = K_F \cd\sigma$.
%   The group
%   $Z_F\cd K'_F$ fixes $x$ and it commutes with the action of
%   $W_F$. Therefore it fixes also the orbit orbit $W_F\cd x$, and hence
%   also $\sigma$.
%  Since $\sigma \subset \liez_F \oplus \liek_F$, $Z_F\cd K'_F$ fixes
%  $\sigma$ and also $K_F \cd \sigma$ pointwise.  Therefore $H_F \cd
%  \sigma = K_F \cd \sigma$.
  So $H_F \cd \sigma=K_F\cd \sigma$ is convex. Since $\ext F = H_F \cd
  x \subset H_F \cd \sigma$, it follows that $F \subset H_F \cd
  \sigma$. On the other hand $\sigma \subset F$ and $F $ is
  $H_F$-invariant, so also $H_F \cd\sigma \subset F$. This establishes
  (ii).
\end{proof}

If $\sigma$ is a face of $P$ set
\begin{gather*}
  G_\sigma : = \{g\ \in \Weyl: g (\sigma ) = \sigma\}.
\end{gather*}

\begin{lemma}
  If $\sigma \in \facesp$ there is a vector $u\in \liet$ that is fixed
  by $G_\sigma$ and such that $\sigma = F_u(P)$.  If $u$ is any
  such vector and $F:=F_u (\c)$, then $F\cap \liet =\sigma$,
  $G_\sigma = W_F\times W'_F$, $\liez_F = \liet^{G_\sigma}$ and $F$
  does not depend on $u$ but only on $\sigma$.
\end{lemma}
\begin{proof}
  The existence of $u$ follows directly from Lemma \ref{u-cono}.  By
  Lemma \ref{facce-inv} (ii) $W_F \times W'_F \subset G_\sigma$, so
  $u\in \liet^{W_F \times W_F'} = \liez_F$ and using Theorem
  \ref{tutte-esposte} it follows that $H_F = C_K(u)$.  Therefore the
  subgroup of $W$ that fixes $u$ is the Weyl group of $(H_F, T)$
  i.e. $W_F\times W'_F$. It follows that $W_F\times
  W'_F=G_\sigma$. From this it follows that $\liez_F =
  \liet^{G_\sigma}$, that $H_F = C_K(\liez_F ) =
  C_K(\liet^{G_\sigma})$ and in particular that $H_F$ and hence $\ext
  F$ and $F$ only depend on $\sigma$.
\end{proof}

Define a map
\begin{gather*}
%  \label{mappa-facce-2}
  \psi : \facesp /\Weyl \ra \faces / K
\end{gather*}
by the following rule: given $\sigma$, fix $u\in \liet^{G_\sigma}$
such that $\sigma = F_u(P)$ and set
\begin{gather*}
  \psi( [\sigma]): = [F_u (\c)].
\end{gather*}
Thanks to the previous lemma $F_u(\c)$ depends only on
$\sigma$, not on $u$. It is clear that $\psi$ is well-defined on
equivalence classes.

\begin{teo}
\label{prova-main}
  The maps $\psi$ and $\phi$ are inverse to each other and
  $\psi([\sigma]) = [ Z_K(\sigma^\perp)\cd \sigma]$.
\end{teo}
\begin{proof}
  Let $\sigma$ be a face of $\polp$. Choose $u\in \liet^{G_\sigma}$
  such that $\sigma = F_u(P)$.  Then $F_u(\c)$ is $T$-stable, so $\phi
  \circ \psi ([\sigma] ) = \phi ([F_u(\c)]) = [F_ u(\c)\cap \liet ] =
  [\sigma]$.  So $\phi \circ \psi$ is the identity.  It follows
  immediately from Theorem \ref{enunciatone} \ref {enunciatone-c} that
  $\phi$ is injective. Hence it is a bijection and $\psi = \phi\meno$.
  By Lemma \ref{medio-simplettico} $\ext F_u (\c)$ is a
  $Z_K(\sigma^\perp)$-orbit.  Hence $K_F \subset Z_K(\sigma^\perp)
  \subset H_F$. By Lemma \ref{facce-inv} (ii) we get $F_u(\c) =
  Z_K(\sigma^\perp) \cd \sigma$.
\end{proof}

\section{Smooth stratification}

\label{strata}

As we saw in the previous section % \S \ref{relazione-politopo}
the group $K$ acts on $\faces$, which is the set of faces of $\c$
% . As shown there,
and this action has a finite number of orbits, which are in one-to-one
correspondence with the orbits of the Weyl group on the finite set
$\facesp$.  Let $B$ denote one of the orbits of $K$ on $\faces$.  We
call $B$ a \emph{face type}. The set
\begin{gather*}
  \SB:= \bigcup_{F \in B} \relint F.
\end{gather*}
is a subset of $\partial \c$, because the faces $F \in B$ are
proper. Since every boundary point lies in exactly one open face
(Theorem \ref{schneider-facce})
\begin{gather*}
  \partial \c = \bigsqcup _{B \in \faces/ K} \SB.
\end{gather*}
We call $\SB$ the \emph{stratum} corresponding to the face type $B$.
The purpose of this section is to show that the strata $\SB$ yield a
stratification of $\c$ in the following sense.

\begin{teo}
  \label{stratification}
  The strata are smooth embedded submanifolds of $\liek$ and are
  locally closed in $\partial
  \c$.  % They give a partition of $\partial \c$.
  For any stratum $\SB$ the boundary $\overline{\SB} \setminus \SB$ is
  the disjoint union of strata of lower dimension.
\end{teo}

There is an obvious map $p:\SB \ra B$ which maps a point $x\in \SB$ to
the unique face $F $ such that $x\in \relint F$.  To study $\SB$ it is
expedient to fix an element $F \in B$. Thus $B=\{ g\cd F: g\in K\}
\cong K/H_F$ and
\begin{gather*}
  \SB = K\cd \relint F =\{g\cd x : g\in K, \, x\in \relint F \}.
\end{gather*}
$K \ra K/H_F$ is a right principal bundle with structure group $H_F$.
Let
\begin{gather*}
  \E = K\times^{H_F} \relint F
\end{gather*}
be the associated bundle gotten from the action of $H_F$ on $\Fo$.
Note that $\E \ra K/H_F$ is a homogeneous bundle in the sense that the
left action of $K$ on $K/H_F$ lifts to an action of $K$ on $\E$ that
is given by the following rule % $a \cd [ g, x] := [ag, x]$
\begin{gather*}
  a \cd [ g, x] := [ag, x] \qquad a,g \in K, \ x\in \Fo
\end{gather*}
(Here $[g,x]$ is the point in the associated bundle.)

\begin{prop}
  \label{bundle}
  Let $B$ be a face type and let $F \in B$ be a representative.
  Define a map
  \begin{gather*}
    f : \E \ra \liek \qquad f( [g,x] ) = g\cd x.
  \end{gather*}
  Then $f$ is a smooth $K$-equivariant embedding of $\E$ into $\liek$
  with image $\SB$.  Therefore $\SB$ is a smooth embedded submanifold
  of $\liek$. Moreover $p: \SB \ra B$ is a smooth fibre bundle.
\end{prop}
\begin{proof}
  It is straightforward to check that $f$ is well-defined, smooth and
  equivariant. It is also clear that $f(\E ) = \SB$.  We proceed by
  showing that $f$ is injective.  Recall from Theorem
  \ref{schneider-facce} that if $F_1$ and $F_2$ are different faces,
  then $\relint F_ 1 \cap \relint F_2 = \vacuo$.  If $f([g, x]) =
  f([g_1, x_1]) $ then $g_1\meno g \cd x= x_1$.  Since $ x_1 \in \Fo$
  and $ g_1 \meno g \cd x \in \relint (g_1\meno g \cd F )$ we get
  $g_1\meno g F =F$, so $[g, x] = [g_1, x_1]$ in $\E$. This shows that
  $f$ is injective.  Next we show that $f$ is an immersion.  Denote by
  $\V$ the fibre of $\E$ over the origin of $K/H_F$.  Since $\E$ is a
  homogeneous bundle and $f$ is equivariant, it is enough to show
  injectivity of $df_p$ at points $p \in \V$, i.e.  at points of the
  form $p=[e,x]$, $x\in \relint F$.  At such points
  \begin{gather*}
    T_p\, \E = T_p \V \oplus U
  \end{gather*}
  with
  \begin{gather*}
    U = \left \{ \desudtzero [\exp(tv) , x ]: v \in \lieh_F^\perp
    \right\}.
  \end{gather*}
  Indeed $T_p\V$ is the vertical space, while $U$ is the tangent space
  at $p$ of a local section of $K\ra K/H_F$.  The injectivity of
  $df_p$ will follow from the following three facts: a) $df_p
  \restr{\V}$ is injective; b) $df_p \restr{U}$ is injective; c) $df_p
  (\V) \cap df_p (U) = \{0\}$.  (a) follows from the fact that
  $f\restr{V}$ is a diffeomorphism of $V$ onto $\relint F$.  To prove
  (b) observe first that if $x\in \Fo$, then $\liek_x \subset
  \lieh_F$. Indeed if $g\in K_x$ then $ g\cd x = x \in \relint (g\cd
  F) \cap \Fo $, so $g\cd F =F$ by Theorem \ref {schneider-facce} and
  $g \in H_F$. Therefore $K_x \subset H_F$ and $\liek_x \subset
  \lieh_F$, as claimed.  Now let $u$ be an element of $U$. By
  definition there is $v \in \liek $ such that
  \begin{gather} \label{uv} u := \desudtzero [\exp(tv) , x ].
  \end{gather}
  Then
  \begin{gather*}
    df_p(u) = \desudtzero f \left ( \left [\exp(tv) , x \right ]
    \right ) = \desudtzero \exp(tv) \cd x = [v,x].
  \end{gather*}
  (The bracket on right is the Lie bracket in $\liek$!)  If $df_p(u)
  =0$, then $[v,x]=0$ and $v \in \liek_x \subset \lieh_F$. Since $v\in
  \lieh^\perp_F$, this means that $v=0$. Thus (b) is proved.  Now
  observe that $[\lieh_F, \lieh_F^\perp ] \subset \lieh_F^\perp$,
  since the adjoint action of $ H_F$ preserves $\lieh_F$ and
  $\lieh_F^\perp$.  If $v\in \lieh_F^\perp$ and $u\in U$ is given by
  \eqref{uv}, then $ df_p(u) = [v, x] \in \lieh_F^\perp $ since $x\in
  F \subset \lieh_F$.  So $df_p(U) \subset \lieh_F^\perp$.  On the
  other hand $ df_p (T_p\V) = T_{f(p)} (\relint F) \subset \lieh_F$.
  It follows that
  \begin{gather*}
    df_p (T_p\V) \cap df_p (U) \subset \lieh_F \cap \lieh_F^\perp =
    \{0\}.
  \end{gather*}
  Thus (c) is proved and $f$ is an immersion.  In order to prove that
  it is an embedding we shall prove that $f$ is proper as a map $f :
  \E \ra \SB=f(\E)$.  Let $\{y_n\}$ be a sequence in $\SB$ converging
  to some point $ y \in \SB$.  Set $[g_n,x_n]:= f\meno(y_n)$.  We wish
  to show that $\{[g_n,x_n] \}$ admits a convergent subsequence.
  Since $K$ is compact by extracting a subsequence we can assume that
  $g_n \to g$.  Then $y_n=f([g_n, x_n]) = g_n \cd x_n$. Therefore $x_n
  = g_n\meno \cd y_n \to x:=g\meno \cd y$. Since $y\in \SB$, $y \in
  \relint (g\meno F)$ and $x\in \relint F$. Therefore $[g_n, x_n] \to
  [g,x]$ as desired.
\end{proof}

\begin{lemma}
  If $B$ is the face type of $F$, then
  \begin{gather*}
    \dim \SB = \dim K - \dim K'_F - \dim Z_F.
  \end{gather*}
\end{lemma}
\begin{proof}
  $\SB$ is a fibre bundle over $K/H_F$ with fibre $\relint F$. Since
  $\dim F = \dim \liek_F $ we get the result.
\end{proof}

We introduce a partial order on the face types, as follows: $B_1
\preceq B_2$ if for some (and hence for any) choice of representatives
$F_i \in B_i$ there is some $g\in K$ such that $g F_1 \subset
F_2$. This is a partial order. We write $B_1 \prec B_2$ if $B_1
\preceq B_2$ and $B_1 \neq B_2$.

\begin{proof}
  [Proof of Theorem \ref{stratification}] We already know that the
  strata are smooth embedded submanifold of $\liek$. In particular
  they are locally closed subsets both of $\liek$ and of $\c$.  By
  Prop. \ref{bundle} $ \SB = f(\E ) = f (K\times^{H_F} \relint F)$. So
  \begin{gather*}
    \overline{ \SB }= f (K\times^{H_F} F) = \bigcup_{F\in B} F.
  \end{gather*}
  Since any face $F$ is the disjoint union of all proper faces
  contained in $F$
  \begin{gather*}
    \overline{ \SB }= \bigcup_{F\in B} \relint F \sqcup \bigsqcup _{C
      \prec B} \bigcup_{G \in C} \relint G = \SB \sqcup
    \bigsqcup_{C\prec B} \mathcal{S}_C.
  \end{gather*}
  To conclude we need to show that $\dim \mathcal{S}_C < \dim \SB$ if
  $C \prec B$. Fix representatives $F \in B$ and $G\in C$ such that $G
  \subsetneq F$.  By the previous lemma it is enough to show that
  $\dim Z_F + \dim K'_F < \dim Z_G + \dim K'_G$.
  % \nuovo{We will prove that $Z_F \cd K'_F \subset Z_G\cd K'_G$. }
  In fact $Z_F \cd K'_F$ fixes $G$ pointwise since $G \subset
  F$. Therefore $Z_F \cd K'_F \subset H_G$. On the other hand if $x
  \in G $, then $\aff (G) = x + \liek_G \subset \aff (F) = x +
  \liek_F$. Hence $K_G \subset K_F$.  It follows that $[\liez_F \oplus
  \liek'_F , \liek_G ] =0 $. Since $\liek_G$ is semisimple, this shows
  that $\liez_F \oplus \liek'_F \perp \liek_G$. But $\liez_F \oplus
  \liek'_F \subset \lieh_G$, so in fact $\liez_F \oplus \liek'_F
  \subset \liez_G \oplus \liek'_G$.  This proves the inequality $ \dim
  Z_F + \dim K'_F \leq \dim Z_G + \dim K'_G$. In the case of equality,
  we would get $ Z_F \cd K'_F = Z_G\cd K'_G$, so $Z_F =Z_G$, $H_F
  =H_G$ and hence $\ext F= \ext G$ and $F=G$.
\end{proof}
\begin{Example}
  We shall describe the strata of the orbitope $\cp$ studied in
  Example \ref{exe1}. We saw there that the faces of $ \cp$ are in
  one-to-one correspondence with subspaces of $\C^{n+1}$. Two
  subspaces are interchanged by an element of $\SU(n+1)$ if and only
  if they have the same dimension. So the orbit types are indexed by
  the dimension.  Let $W \subset \C^{n+1}$ be a subspace of dimension
  $k$, let $F = \conv (\PP(W))$ be the corresponding face and let $B$
  be the orbit type of $F$. Then
  \begin{gather*}
    B \cong K / H_F = \SU(n+1) / \operatorname{S}(\operatorname{U}(W)
    \times \operatorname{U}(W^\perp) ).
  \end{gather*}
  Therefore $B$ is simply the Grassmannian $\grass(k, n+1)$. Since $\relint
  F = \{A\in F: \operatorname{rank} A = k\}$, it follows that
  \begin{gather*}
    \SB = \{ A\in \mathcal{H}_1 : A \geq 0,\ \operatorname{rank} A=k\}.
  \end{gather*}
  In fact this is a bundle over the Grassmannian of $k$-planes.
  Finally, notice that $H_F$ acts on $\relint F$ simply by the adjoint
  action of $\SU( W)$.
\end{Example}

\section{Satake combinatorics of the faces}
\label{satake-section}

In this section we describe the faces of $\c$ and the faces of the
momentum polytope in terms of root data.  The description uses the
notion of $x$-connected subset of simple roots, which was introduced
in \cite{satake-compactifications}. In that paper Satake introduced
certain compactifications of a symmetric space of noncompact type (the
Satake-Furstenberg compactifications). The notion of $x$-connected
subset was used in the study of the boundary components of these
compactifications.  It is no coincidence that faces of $\c$ and
boundary components admit a description in terms of the same
combinatorial data: in fact it was shown in \cite{biliotti-ghigi-2}
that the Satake compactifications of the symmetric space $K^\C /K$ are
homeomorphic to convex hulls of integral coadjoint orbit of $K$.  Here
we do not use the link with the compactifications. Instead we show
directly how to construct all the faces of $\c$ (up to conjugation)
starting from the root data. This is accomplished for a general
coadjoint orbit with no integrality assumption.

Fix a maximal torus $T$ of $K$ and a system of simple roots $\Pi
\subset \roots= \roots (\liek^\C, \liet^\C)$.  As usual we identify
$\liek^\C$ with its dual using the Killing form $B$.  The roots get
identified with elements of $i\liet$.
\begin{defin}
  A subset $E\subset i \liet$ is \emph{connected} if there is no pair
  of disjoint subsets $D,C\subset E$ such that $D\sqcup C =E$, and
  $\sx x,y \xs=0$ for any $x \in D$ and for any $y \in C$.
\end{defin}
(A thorough discussion of connected subsets can be found in \cite[\S
5]{moore-compactifications}.)  Connected components are defined as
usual. For example the connected components of $\simple$ are the
subsets corresponding to the simple roots of the simple ideals in
$\liek$.

\begin{defin}
  If $x $ is a nonzero vector of $\liet$, a subset $I \subset\simple$
  is called $x$-\enf{connected} if $I\cup\{i x\}$ is connected.
\end{defin}
Equivalently $I \subset \simple$ is $x$-connected if and only if every
connected component of $I$ contains at least one root $\alfa$ such
that $\alfa (x) \neq 0$. \changed{By definition the empty set is
  $x$-connected.}

\begin{defin}
  If $I\subset \simple$ is $x$-connected, denote by $I'$ the
  collection of all simple roots orthogonal to $\{i x\}\cup I$.  The
  set $J:=I\cup I'$ is called the $x$-\enf{saturation} of $I$.
\end{defin}
The largest $x$-connected subset contained in $J$ is $I$. So $J$ is
determined by $I$ and $I$ is determined by $J$.  Given a subset
$E\subset \simple$ we will use the following notation:
\begin{gather*}
  \liet_E :=  \liet \cap \bigcap_{\alfa \in E} \ker \alfa \\
  \root_E = \root \cap \spam_\R (E) \qquad \root_{E,+} =\root_E \cap
  \root_+
  \\
  \liet^E = \sum_{\alfa \in E } \R i H_\alfa = \text { orthogonal
    complement of }
  \liet_E \text { in } \liet\\
  \lieh_E : = \liet \oplus \bigoplus_{\alfa \in \root_{E,+}} Z_\alfa
  \quad \liek_E : = \liet^E \oplus \bigoplus_{\alfa \in \root_{E,+}}
  Z_\alfa.
  % \qquad \lieh_E^\C = \liet^\C \oplus \bigoplus_{\alfa \in \root_E}
  % \lieg_\alfa.
\end{gather*}
We denote by $T_E$, $H_E$, $K_E$ the corresponding connected
subgroups.  Note that $H_E$ is the subgroup associated to the subset
$E \subset \simple$, while $H_F$ is the subset associated to the face
$F \subset \c$. This should cause no confusion.

\begin{lemma}
  \label{parabolico-1}
  Let $\OO$ be a full coadjoint orbit and let $F \subset \c$ be a
  proper face. Assume that $u\in \CF$ and that $v\in \CF\cap
  \liez_F$. Let $\alfa \in \root$.
  \begin{enumerate}
  \item If $\alfa(u) = 0$, then $\alfa(v) =0$.
  \item If $-i\alfa(u) > 0$, then $-i\alfa(v) \geq 0$;
  \end{enumerate}
\end{lemma}
\begin{proof}
  (a) $Z_K(u) \subset H_F$, since $F = F_u(\c)$, and $H_F =Z_K(v)$ by
  Theorem \ref{tutte-esposte}. If $\alfa(u) =0$, then $Z_\alfa \subset
  \liez_\liek(u) \subset \lieh_F = \liez_\liek(v)$, hence
  $\alfa(v)=0$. (b) Assume by contradiction that $-i\alfa(v) <0$. Set
  $u_t = (1-t) u + t v$. By Proposition \ref{u-cono} $\CF$ is convex,
  so $u_t \in \CF $ for any $t\in [0,1]$.  Since $-i\alfa(u_0) >0$ and
  $-i\alfa(u_1) <0$, there is some $s\in (0,1)$ such that
  $\alfa(u_s)=0$. Since $u_s \in \CF$ and $\alfa(v) \neq 0$, this
  would contradict (a).
\end{proof}
Denote by $\chamber$ the positive Weyl chamber associated to
$\simple$.  The following is immediate and well-known.
\begin{lemma}
  \label{parabolico-2}
  If $v\in \cchamber$, then $\liez_\liek(v) = \lieh_E$ with
  $E=\{\alfa\in \simple: \alfa(v) = 0 \}$.
\end{lemma}

\begin{teo}
  \label{satakone}
  Let $\OO$ be a full coadjoint orbit and let $x $ be the unique point
  in $\OO \cap \cchamber$.
  \begin{enumerate}
  \item If $I \subset \simple$ is $x$-connected and $J$ is its
    $x$-saturation, then
    \begin{gather*}
      F:= \conv ( H_J \cd x)
    \end{gather*}
    is a face of $\c$. If $u \in \liet_J$ and $-i\alfa(u) >0 $ for any
    $\alfa \in \simple \setminus J$, then $F=F_u (\c)$. Moreover
    \begin{gather}
      \label{HFHJ}
      H_F = H_J \qquad Z_F = T_J \qquad K_F =K_I \qquad K'_F = K_{I'}.
    \end{gather}
  \item Given an arbitrary subset $E \subset \simple$, denote by $I$
    the largest $x$-connected subset contained in $E$ and by $J$ the
    $x$-saturation of $I$.  Then $H_E\cd x = H_I \cd x = H_J \cd x$.
  \item Any face of $\c$ is conjugate to one of the faces constructed
    in (a). More precisely, given a face $F$ and a maximal torus $T
    \subset H_F$ there are a base $\simple \subset \root (\liek^\C,
    \liet^\C)$ and a subset $I\subset \simple$ with the following
    properties: (i) if $\chamber$ is the positive Weyl chamber
    corresponding to $\simple$, then $\cchamber\cap \ext F \neq
    \vacuo$; (ii) if $x$ is the unique point in $\cchamber \cap \ext
    F$, then $I$ is $x$-connected and $F= \conv (H_J\cd x)$, where $J$
    is the $x$-saturation of $I$.
  \end{enumerate}
\end{teo}
\begin{proof}
  (a) Since the set $\{\alfa \restr{\liet_J} : \alfa \in \simple
  \setminus J \}$ is a basis of $\liet_J^*$, we can pick $u\in
  \liet_J$ such that $\alfa (u) >0$ for any $\alfa \in \simple
  \setminus J$. Then $Z_K(u) =H_J$.  Set $F:=F_u (\c)$. We claim that
  $x\in F$. Indeed $x$ and $u$ belong to $\cchamber$, so by Lemma
  \ref{hesso-2} $x$ is a maximum point of $\mom_u$, i.e. $x\in \ext
  F$.  By Lemma \ref{hesso} $\ext F = Z_K(u) \cd x$, so $F = \conv
  (H_J\cd x)$.  This proves that $\conv (H_J\cd x)$ is indeed a face
  of $\c$.  By Lemma \ref{inclusioni} $K_F \subset H_J=Z_K(u) \subset
  H_F$ and $K_F= K_I$.  Pick $v\in \CF \cap \liez_F$ (this exists by
  Theorem \ref{tutte-esposte}). By Lemma \ref{parabolico-1}
  $-i\alfa(v) \geq 0$ for every $\alfa \in \roots_+$, i.e.  $v\in
  \cchamber$. By Theorem \ref{tutte-esposte} (c) and Lemma
  \ref{parabolico-2} $ \lieh_F = \liez_\liek(v) = \lieh_E $, where
  $E=\{\alfa \in \simple: \alfa(v) =0 \}$. We claim that $E=J$. Indeed
  $\lieh_J \subset \lieh_F = \lieh_E$, so $J \subset E$. If we write
  $E = I \sqcup E'$, then $I' \subset E'$. Conversely, if $\alfa \in
  E'$, then $Z_\alfa \perp \liek_I = \liek_F$ (simply because the root
  space decomposition is orthogonal), so $Z_\alfa \subset
  \liek'_F$. This entails on the one hand that $[Z_\alfa,\liek_I]=0$,
  i.e. $\alfa \perp I$; on the other hand that $Z_\alfa$ fixes $x$,
  i.e. $\alfa(x)=0$. This means in fact that
  $\alfa \in I'$.  Hence $E=J$ as claimed and \eqref{HFHJ} follow. \\
  (b) Split $E$ in connected components: $E = E_1 \sqcup \cds \sqcup
  E_r$. We can assume that $E_j$ is $x$-connected iff $j \leq q$ for
  some $q$ between $1$ and $ r$. Then $I =E_1 \sqcup \cds \sqcup E_q$.
  Set $E' := E \setminus I = \sqcup_{j > q} E_J$. Then clearly $E'
  \subset I'$. So $E \subset J$.  Let $F=\conv( H_j\cd x)$ be the face
  constructed from $J$ as in (a). Then $H_F=H_J$ and $K_F =
  K_I$. Since $ I\subset E \subset J$, $K_I \subset H_E \subset
  H_J$. But $K_I \cd x = K_F
  \cd x = H_F\cd x =H_J\cd x$, so $H_E\cd x = H_J \cd x$ as desired.\\
  (c) If $F=\c$, then $F=\conv(H_J)$ with $I=J=\simple$.  Otherwise $F
  $ is a proper face. Fix a point $x\in \ext F \cap \liet$.  By
  Theorem \ref{tutte-esposte} (b) there is a vector $u\in \liez_F$
  such that $F=F_u(\c)$. Then $\ext F = \ml(\mom_u)$, so there is a
  Weyl chamber $\chamber$ such that $x, u \in \cchamber$. Let $\simple
  $ be the base corresponding to $\chamber$.  By Theorem
  \ref{tutte-esposte} (c) $H_F =Z_K(u)$. Since $u\in \cchamber$, Lemma
  \ref {parabolico-2} says that $H_F =H_E$ with $E=\{\alfa \in
  \simple: \alfa(u) =0\}$. Let $I$ and $J$ be as in (b). Then $I$ is
  $x$-connected and using (b) we get $\ext F = H_F\cd x =H_E\cd x =
  H_J\cd x $. Thus $F=\conv (H_J\cd x)$ as desired.
\end{proof}

\begin{remark}
  In the proof of (c) we have in fact that $E=J$. Indeed
  % $H_F=Z_K(u)=H_E$. But
  from (a) $H_F=H_J$, so $H_E=H_J$ i.e. $E=J$.
\end{remark}
\begin{Example}
  Let $K=\SU (n+1)$, $n\geq 4$, and let $x \in \mathfrak{su}(n+1)$ be
  the diagonal matrix $ x = \diag (i (n-1),i(n-1),-2i,\ldots,-2i)$.
  The coadjoint orbit through $x$ is the momentum image of the
  Grassmannian $\grass(2, n+1)$.  Let $\liet$ be the set of the
  diagonal matrices and denote by $\Pi=\{ \alpha_1,\ldots, \alpha_n
  \}$ the standard set of simple roots, i.e.  $\alpha_i (\diag (x_1,
  \lds, x_{n+1}))=x_i - x_{i+1}$.  The vector $x$ lies in the closure
  of the positive Weyl chamber containg $x$ and $\alpha_i( x) \neq 0$
  if and only if $i =2$.  Therefore the $x$-connected subsets of $\Pi$
  are the following:
  \begin{enumerate}
  \item $I^1_k = \{ \alpha_1, \alpha_2, \ldots, \alpha_k \}$, $2\leq k
    \leq n$;
  \item $I^2_k =\{ \alpha_2, \ldots, \alpha_k \}$, $2\leq k \leq n$;
  \item \changed{the empty set}.
  \end{enumerate}

  \changed{For $I=\vacuo$, $\Delta_I = \vacuo$, $H_I = T$ and the
    $x$-saturation $J$ of $I$ consists of the simple roots that are
    orthogonal to $ix$. Therefore $H_J = Z_K(x)$ and $H_J\cd x =
    \{x\}$. The corresponding face is the vertex $F=\{x\}$.}

  For $i=1,2$ let $J^i_k$ be the $x$-saturation of $I^i_k$ and set
  $F^1_k = \mathrm{conv}(H_{J^1_k } \cdot x)$.  It is easy to check
  that $J^1_k= I^1_k \bigcup \{ \alpha_{k+2} , \ldots , \alpha_n \}$.
  $K_{I^1_k} $ is the image of the embedding
  \[
  \SU(k+1) \hookrightarrow \SU(n+1) \qquad A \mapsto
  \left(\begin{array}{cc} A & 0 \\ 0 & \mathrm{Id}
      \\ \end{array}\right)
  \]
  and $H_{J^1_k }= \mathrm{S}(\mathrm{U}(k+1)\times \mathrm{U}(n-k))$.
  Hence \[ \ext F^1_k=K_{I^1_k} \cdot x= \SU(k+1)/
  \mathrm{S}(\mathrm{U}(2) \times \mathrm{U}(k-1)),
  \]
  is the complex Grassmannian $\grass(2,k+1)$.  The stratum
  corresponding to $F^1_k$ is a fibre bundle over $\SU(n+1)/
  \mathrm{S}(\mathrm{U}(k+1)\times \mathrm{U}(n-k)) = \grass(k+1,
  n+1)$.

  The $x$-saturation of $I^2_k$ is $J^2_k= I^2_k \bigcup \{
  \alpha_{k+2} , \ldots , \alpha_n \}$.  $K_{I^2_k}$ is the image of
  the embedding
  \begin{gather*}
    \SU(k) \hookrightarrow \SU(n+1) \qquad A \mapsto
    \left(\begin{array}{ccc} 1 & 0 & 0 \\ 0 & A & 0 \\ 0 & 0 &
        \mathrm{Id} \\ \end{array}\right).
  \end{gather*}
  $H_{J^2_k }= \mathrm{S}(\mathrm{U}(1) \times \mathrm{U}(k)\times
  \mathrm{U}(n-k))$ and
  \[
  \ext F^2_k=K_{I^2_k} \cdot x= \SU(k)/ \mathrm{S}(\mathrm{U}(1)
  \times \mathrm{U}(k-1))
  \]
  is a complex projective space $\PP^{k-1}(\C)$. The strata
  corresponding to $F^2_k$ is a fibre bundle over the flag manifold
  $\SU(n+1)/ \mathrm{S}(\mathrm{U}(1) \times \mathrm{U}(k)\times
  \mathrm{U}(n-k))$.
\end{Example}
\section{Complex geometry of the faces }
\label{complex-structure}

In the previous sections we have described the faces of $\c$ in terms
of their extreme sets $\ext F$ and have caracterized the submanifolds
$\ext F \subset \OO$ in various ways. Here we wish to prove Theorem
\ref{main-2}, which amounts to the equivalence between (a) and (b) in
Theorem \ref{parabolique} below. This will add another
characterization in terms of the complex structure of $\OO$.

\begin{teo}
  \label{parabolique}
  Let $\OO ' \subset \OO$ be a submanifold.  The following conditions
  are equivalent.
  \begin{enumerate}
  \item $\OO'$ is a compact orbit of a parabolic subgroup of $G$.
  \item There is a face $F$ of $\c$ such that $\OO' = \ext F$.
  \item $\OO'$ is compact and the subgroup
    \begin{gather}
      \label {def-P}
      P:=\{g\in G: g \cd \OO' = \OO'\}
    \end{gather}
    is a parabolic subgroup of $G$ that acts transitively on $\OO'$;
  \item There are a maximal torus $T\subset K$, a Weyl chamber
    $\chamber \subset \liet$ and a subset $E$ of the corresponding set
    of simple roots $\Pi$ such that $\OO'\cap \cchamber \neq \vacuo$
    and $\OO'$ is an orbit of $H_E$.
  \end{enumerate}
\end{teo}
\begin{proof}
  That (d) is equivalent to (b) is the content of Theorem
  \ref{satakone}.
  \\
  \noindent
  (a) $\Rightarrow $ (c) Since $\OO'$ is an orbit of some parabolic
  subgroup $Q$, the subgroup $P$ contains $Q$ so it is parabolic.
  \\
  \noindent
  (c) $\Rightarrow $ (d) Since $P$ is parabolic we can find a maximal
  torus $T\subset K$ and a system of simple roots in $\liet$ in such a
  way that $B_- \subset P$. So $B_-$ acts on $\OO'$ and by the Borel
  fixed point theorem $B_-$ has some fixed point $x\in \OO'$. Since
  $x$ is fixed by $T\subset B_-$, $x\in \liet$ and it follows from
  Lemma \ref{Borel} that $x\in \cchamber$.  If $E\subset \simple$ set
  \begin{gather*}
    \liu_E:= \bigoplus_{\alfa \in \root_- \setminus \root_E}
    \lieg_\alfa \qquad \liep_E:= \liet^\C \oplus\bigoplus_{\alfa \in
      \root_- \cup \setminus \root_E} \lieg_\alfa.
  \end{gather*}
  Then $ \liep_E = \lieh_E^\C \oplus \liu_E$ is a parabolic
  subalgebra.  Denote by $U_E$ and $P_E$ the corresponding connected
  subgroups of $G$.  Then $P_E$ is a parabolic subgroup, $U_E$ is its
  unipotent radical and $H_E^\C$ is a Levi factor. In particular $P_E
  = H_E^\C \cd U_E$ and $U_E \lhd P_E$.  Since $B_-\subset P$ there is
  some $E\subset \simple$ such that $P=P_E$.  Since $U_E \subset B_-
  \subset G_x$ we conclude that $\OO'=P_E\cd x = H_E^\C \cd x$. As
  $\OO'$ is compact, the compact form $H_E$ must be transitive on
  $\OO'$. This concludes the proof.
  \\
  \noindent
  (d) $\Rightarrow $ (a) First observe that $\OO'=H_E\cd x$ is a
  complex submanifold since it is a connected component of the fixed
  point set of the torus $T_E$.  Therefore $H_E^\C$ preserves $ \OO'$.
  % Set
  % \begin{gather*}
  %   \liu_E:= \bigoplus_{\alfa \in \root_- \setminus \root_E}
  %   \lieg_\alfa \qquad \liep_E:= \liet^\C \oplus\bigoplus_{\alfa \in
  %     \root_- \cup \setminus \root_E} \lieg_\alfa.
  % \end{gather*}
  % Then $ \liep_E = \lieh_E^\C \oplus \liu_E$.  Denote by $U_E$ and
  % $P_E$ the corresponding connected subgroups of $G$.  Then $P_E$ is
  % a parabolic subgroup, $U_E$ is its unipotent radical and $H_E^\C$
  % is a Levi factor. In particular $P_E = H_E^\C \cd U_E$ and $U_E
  % \lhd P_E$.
  By assumption there is $x\in \cchamber \cap \OO'$. By Lemma
  \ref{Borel} the stabilizer $G_x$ contains the negative Borel
  subgroup, so $U_E$ fixes $x$.  If $x'\in \OO'$, there is $a \in H_E$
  such that $x'=a\cd x $. If $b\in U_E$ then $a\meno ba \in U_E$, so
  $a \meno b a \cd x = x$ and $b \cd x' = b a \cd x = a\cd x =
  x'$. Hence $U_E$ fixes pointwise $\OO'$.  Therefore $P_E$ preserves
  $\OO'$ which is therefore a compact $P_E$-orbit.
\end{proof}

We notice that in condition (d) the set $E$ can be chosen to be the
$x$-saturation of the maximal $x$-connected subset $I\subset E$ as
shown in Theorem \ref {satakone} (c).

The above result establishes a one-to-one correspondence between two
rather distant classes of objects: on the one side the faces of the
orbitope $\c$, on the other side the closed orbits of parabolic
subgroups of $G$ inside $\OO$. To illustrate this correspondence
recall the following fact.

\begin{lemma}
  If $P \subset G$ is a parabolic subgroup, in $\OO$ there is only one
  orbit of $P$ which is closed.
\end{lemma}
\begin{proof}
  Since the action is algebraic and $\OO$ is a compact manifold, there
  is at least one orbit which is closed. Let $\OO'\subset \OO$ be a
  closed $P$-orbit and let $B\subset P$ be a Borel subgroup. Then
  $\OO'$ is $B$-invariant, so it contains a closed $B$-orbit. But the
  $B$-orbits in $\OO$ are just the Schubert cells and the only one
  which is closed is the fixed point of $B$. Hence any closed
  $P$-orbit contains this fixed point and this implies that the closed
  $P$-orbit is unique.
\end{proof}

The above uniqueness statement can also be considered from the point
of view of the orbitope, as can be seen from the proof of the
implication (c) $\Rightarrow $ (d) in the previous theorem. Indeed, if
$P$ is a parabolic subgroup, we write it as $P=P_E$ for some $E\subset
\simple$.  Then there is a unique orbit of $H_E$ that is of the form
$\ext F$, namely the orbit $H_E\cd x$ for $x\in \OO \cap \cchamber$.
Alternatively this orbit can be described as follows: choose $u\in
\liet_E = \liez(\lieh_E)$ such that $-i\alfa(u) > 0 $ for $\alfa \in
E$. Then the closed $P$-orbit is $ \ml(\mom_u)$. In a sense to fix a
parabolic subgroup $P_E$ is equivalent to fixing $H_E$ and the vector
$u$. So once $P_E$ is fixed we know both $H_E$ and which component of
$\OO\cap \lieh_E$ corresponds to the maximum of $\mom_u$.

To conclude we wish to interpret geometrically condition (c) of
Theorem \ref{parabolique}. Let $\OO'$ be a complex submanifold of
$\OO$. Let $\hilbo$ denote the Hilbert scheme of the projective
manifold $\OO$. If $Y \subset \OO$ is a subscheme, let $[Y]$ be its
Hilbert point. (See e.g.  \cite[Chapter
IX]{arbarello-cornalba-griffiths-2}.)  The group $G$ acts on $\hilbo$
by sending the Hilbert point $[Y]$ of a subscheme $Y\subset \OO$ to
$[g\cd Y]$.

\begin{prop}
\label{cicli}
  Let $\OO' \subset \OO$ be a complex submanifold which is an orbit of
  some subgroup of $K$.  Let $f : G \ra \hilbo$ be the map $f(g) : =
  [g\cd \OO']$. Then the following conditions are all equivalent to
  condition (c) of Theorem \ref{parabolique}:
  \begin{enumerate}
  \item [i)] $f(G)$ is compact;
  \item [ii)] $f(K)$ is a subscheme of $\hilbo$;
  \item [iii)] $f(G) = f(K)$.
  \end{enumerate}
\end{prop}
\begin{proof}
  $f(G)$ is just the orbit of $G$ through the point $p=[\OO'] \in
  \hilbo$, while $f(K)$ is the orbit of $K$ through $p$.  The subgroup
  $P$ defined in \eqref{def-P} is just the stabilizer $G_p$. Therefore
  $f(G) \cong G /P$. It follows immediately that the three conditions
  are equivalent to $P$ being parabolic, so they are implied by
  (c). Conversely, if they are satisfied, $P$ is parabolic. By
  assumption $\OO'$ is an orbit of some subgroup $L\subset K$. Then
  $L\subset P$ and $\OO'$ is a $P$-orbit, thus (c) holds.
\end{proof}

\begin{Example}
  Consider the orbitope of $\PP^2(\C)$ as described in Example
  \ref{exe1}.  The complex lines satisfy the conditions in the
  proposition and in fact they do generate faces of $\c$: if
  $\OO'\subset \PP^2(\C)$ is a line the set $\conv (\OO')$ is a face
  of $\c$. Also plane conics are complex submanifolds of $\PP^2(\C)$
  that are homogenous for a subgroup of $\Sl(3, \C)$, namely
  $\operatorname{SO}(3,\C)$. Nevertheless the orbit of $\Sl(3,\C)$
  through a conic is not compact since smooth conics degenerate to
  singular ones. So conics do no satisfy the conditions above and in
  fact conics do not generate faces of $\c$.
\end{Example}

\begin{Example}
  Let $L \subset K$ be the centralizer of a torus and let $\OO'
  \subset \OO$ be an orbit of $L$.  As we have shown in general the
  set $F = \conv (\OO')$ is not a face of $\c$. One condition is that
  $\OO'\subset \liel$. In fact if $L=Z_K(u)$, and $F=F_u(\c)$, then
  $\OO'= \ext F = \ml(\mom_u) \subset \Crit(\mom_u) = \OO\cap \liel$.
  This condition is not enough either.  In fact $\Crit(\mom_u)$ will
  contain at least two orbits, one for the maximum and one for the
  minimum. These are ''good'' orbits, in the sense that they
  correspond to faces, namely to $F_u(\c)$ and $F_{-u}(\c)$
  respectively.  The orbits in between in general do not generate
  faces.  Consider the following example. Let $\OO \subset \su(3) $ be
  the momentum image of the flag manifold of pairs $(L_1, L_2 ) $
  where $L_1 \subset L_2 \subset \C^3$ and $\dim L_i = i$.  Let
  $u=i\operatorname{diag} (1, 1 , -2)$.  Set $V=\C^2 \times \{0\}$.
  Then $ \Crit(\mom_u) $ has the following three connected components:
  \begin{gather*}
    C_1 = \{(L_1, L_2) \in \OO : L_1 \in \PP(V) , L_2 = L_1 \oplus \C e_3\}\\
    C_2 = \{(L_1, L_2) \in \OO : L_1 \subset  L_2 \subset V\}\\
    C_3 = \{(L_1, L_2) \in \OO : L_1 = \C e_3 \}.
  \end{gather*}
  Each component is an orbit of $Z_K(u) = S \left(U(2) \times U(1)
  \right )$. Let $P_i$ denote the stabilizer of $C_i$ for the action
  of $G=\Sl(3,\C)$.  Then $P_2 = \{ g \in \Sl(3, \C): g (V)=V\}$ and
  $P_3=\{ g : ge_3 = e_3\}$. These two subgroups are parabolic. So
  $C_2$ and $C_3$ correspond to faces, by Prop. \ref{parabolique}.  On
  the other hand we claim that $P_1$ is the subgroup of $\Sl(3, \C)$
  of matrices of the form
  \begin{gather*}
    g = \begin{pmatrix}
      A & 0 \\
      0& \la
    \end{pmatrix} \qquad A\in \Sl(2,\C), \la \in \C^*.
  \end{gather*}
  It is clear that matrices of this form lie in $P_1$. Conversely
  assume $g\in P_1$.  Then $g(V) = V$.  Write $ge_3= \la e_3 + w$ with
  $w\in V$. For any $v\in V \setminus \{0\}$ the plane $g \spam (v,
  e_3) = \spam (gv, ge_3) $ contains $e_3$. Hence $w \in \spam(gv,
  e_3)$. Since  $v \in V\setminus \{0\}$ is arbitrary it follows that $w=0$.
 The claim is proved, hence $P_1$ is not parabolic and $\conv (C_1)$ is not a face of $\c$.
\end{Example}

\section{The case of an integral orbit}
\label{sezione-caso-intero}

% Before defining integrality of a coadjoint orbit it is proper to
% recall the difference between \emph{adjoint orbits} and \emph
% {coadjoint} ones. We have fixed from the beginning a scalar product
% on $\liek$ gotten from the Killing form. Via the induced isomorphism
% $\liek \cong \liek^*$ the coadjoint and the orbits become the same.
%
% The definition of the KSS form is best given in the coadjoint
% picture. There it is
%
%
% Once we decide that an adjoint orbit is mapped isomorphically onto a
% coadjoint one via the isomorphism $\liek \cong \liek^*$ induced by
% the Killing form, we have a canonical symplectic form and it makes
% sense to ask for integrality.

A coadjoint orbit $\OO \subset \liek$ is \emph{integral} if
$[\om]/2\pi$ lies in the image of the natural morphism $H^2(\OO,
\Zeta) \ra H^2(\OO, \R)$. (Here $\om$ is the Kostant-Kirillov-Souriau
form.)  If $\OO$ is integral there is a complex line bundle $L \ra
\OO$ such that $[\om]=2\pi \chern_1(L)$. This line bundle can be made
$K$-equivariant and holomorphic with respect to the structure $J$ on
$\OO$ and it supports a unique $K$-invariant Hermitian bundle metric
$h$ such that $\om=iR(h)$.  With this holomorphic structure the line
bundle $L$ turns out to be very ample. Set $V := (H^0(\OO, L)
)^*$. Then $V$ inherits from $\om$ and $h$ an $L^2$-scalar
product. Moreover $V$ is an irreducible representation of $K$ and
there is a unique orbit $M \subset \PP(V)$ which is a complex
submanifold of $\PP(V)$.  This orbit is simply connected.  Fix on $M$
the restriction of the Fubini-Study form gotten from the $L ^2$-scalar
product on $V$.  Since $K$ is semisimple there is a unique momentum
map $\mom : M \ra \liek$ and $\OO = \mom(M)$. Conversely, if there is
an irreducible $K$-representation $V$ such that $\OO=\mom (M)$ for the
unique complex orbit $M\subset \PP(V)$, then $\OO$ is integral.  This
follows from the fact that the momentum map $\mom: M \ra \OO$ is a
symplectomorphism.

Another way to express integrality of $\OO$ is the following.  Fix a
maximal torus $T\subset K$ and choose a point $x\in \OO \cap \liet$.
Recall that a linear functional $\lambda \in (i\liet)^*$ is an
\emph{algebraically integral weight} if
\begin{gather*}
  \frac{\sx \lambda, \alfa\xs }{|\alfa|^2} =
  \frac{\lambda(H_\alfa)}{|H_\alfa|^2} \in \Zeta
\end{gather*}
for any root $\alfa \in \roots(\liek^\C, \liet^\C)$, see e.g.
\cite[p. 265]{knapp-beyond}. % or \cite [p. 130]{sepanski}
Then $\OO$ is integral if and only if $\la= \sx ix, \cdot \xs $ is an
algebraically integral weight. (For all this see \cite[Chapter
1]{kirillov-lectures} or \cite{kostant-quantization}.)

% We wish to show that if $\OO \subset \liek $ is an integral
% coadjoint orbit then $\ext F$ is integral for any face $F \subset
% \c$.  But we need to say explicitely how to

\begin{teo}
  Let $\OO \subset \liek $ be an integral coadjoint orbit and let $F$
  be a face of $\c$. Write $\ext F = x_0 + K_F \cd x_1$ as in
  \eqref{extfx0}.  Denote by $\sx \, , \ \xs_F$ the scalar product on
  $\liek_F$ induced by the Killing form of $\liek_F$.  Define $x'_1
  \in \liek_F$ by the following rule:
  \begin{gather}
    \label{defxprimo}
    \sx x_1 ', y \xs_F = \sx x_1, y \xs
  \end{gather}
  for all $y\in \liek_F$.  Then $K_F \cd x_1'$ is an integral
  coadjoint orbit in $\liek_F$.
\end{teo}
\begin{proof}
  This fact can be proved in a variety of ways using the various
  characterizations of integrality.  One simple way is using the
  definition, i.e. the condition on the integrality of the
  Kostant-Kirillov-Souriau form.  Let $\om_F$ be the KSS form of
  $K_F\cd x'_1 \subset \liek_F$.  Let $\mu \in \liek_F^*$ be the
  functional $\mu(y) = \sx x_1, y\xs = \sx x_1 ', y \xs_F $. The
  stabilizers (for the adjoint action) of $x_1 $ and $x_1' $ are the
  same, because both coincide with the stabilizer of $\mu$ (for the
  coadjoint action).  Moreover the stabilizers in $K_F$ of $x$ and of
  $x_1$ coincide since $x=x_0 + x_1$ and $x_0$ is fixed by
  $K_F$. Summing up we get that the stabilizers in $K_F$ of $x_1'$ and
  $x$ coincide.  Hence the map
  \begin{gather*}
    j: K_F\cd x'_1 \hookrightarrow \liek \qquad g \cd x'_1 \mapsto
    j(g\cd x'_1) := g\cd x
  \end{gather*}
  is an embedding of $K_F \cd x'_1$ onto $ \ext F = K_F \cd x \subset
  \liek$.  We claim that $j^*\om = \om_F$.  By equivariance it is
  enough to check that $j^*\om = \om_F$ at $x'_1$. Take $X, Y \in
  \liek_F$ and set $u=[X, x'_1], v= [Y, x'_1]$. Then
  \begin{gather*}
    dj_{x'_1}(u) = \desudtzero j \left (\Ad \left ( \exp tX \right )
      x'_1 \right ) = \desudtzero \left (\Ad \left ( \exp tX \right )
      x \right ) = [X, x]
  \end{gather*}
  and similarly $dj_{x'_1}(v) = [Y,x]$. Hence $ j^*\om (u,v) = \om (
  [X, x], [Y,x] ) = \sx x , [X,Y] \xs $. Since $[X,Y] \in \liek_F$ and
  $x_0 \in \liez_F$, $x_0 \perp [X,Y]$. Therefore $\sx x , [X,Y] \xs =
  \sx x_1, [X,Y] \xs = \sx x'_1, [X,Y] \xs_F = \om_F(u,v)$. This
  proves that indeed $\om_F= j^*\om$ and thus $[\om_F] /2\pi$ is
  integral if $[\om]/2\pi$ is.

\end{proof}

\begin{remark}
  Since the various definitions of integrality are equivalent, this
  theorem ensures that if $\sx i x, \cdot \xs $ is an integral weight,
  then $\sx i x_1 , \cdot \xs_F$ is integral as well. Since integral
  weights give rise to representations, to each face $F$ of an
  integral coadjoint orbitope is attached an irreducible
  representation of $K_F$. If one fixes root data and $F$ is the face
  corresponding to an $x$-connected subset $I \subset \Pi$ as in \S
  \ref{satake-section}, then the representation corresponding to $F$
  is the representaion $V_I$ originally described by Satake
  \cite[p. 89] {satake-compactifications} (see also
  \cite[p. 67]{borel-ji-libro}.
\end{remark}

\begin{remark}
  If $\OO$ is an integral orbit, then $\OO$ is the momentum image of a
  flag manifold $M$ provided with an invariant Hodge metric lying in a
  polarization $L \ra M$. The space $H^0(M, L)$ is an irreducible
  representation $\tau$ of $K$. Out of these data one can construct a
  Satake-Furstenberg compactification $\overline{X_\tau}^S$ of the
  symmetric space $K^\C /K$ and it is possible to define a
  homeomorphism (named after Bourguignon-Li-Yau) between this
  compactification and the orbitope $\c$. This was accomplished in
  \cite{biliotti-ghigi-2}. Since this homeomorphism respects the
  boundary structure, some properties of the faces of $\c$ can be
  deduced in this way. The arguments in the present paper apply also
  to the non-integral case, give much more information and are more
  direct and geometric, since no use is made of the Bourguignon-Li-Yau
  map.
\end{remark}

\def\cprime{$'$}


\begin{thebibliography}{10}

\bibitem{arbarello-cornalba-griffiths-2}
E.~Arbarello, M.~Cornalba, and P.~A. Griffiths.
\newblock {\em Geometry of algebraic curves. {V}ol. {II}}, volume 268 of {\em
  Grundlehren der Mathematischen Wissenschaften [Fundamental Principles of
  Mathematical Sciences]}.
\newblock Springer-Verlag, New York, 2011.

\bibitem{atiyah-commuting}
M.~F. Atiyah.
\newblock Convexity and commuting {H}amiltonians.
\newblock {\em Bull. London Math. Soc.}, 14(1):1--15, 1982.

\bibitem{audin-torus-actions}
M.~Audin.
\newblock {\em Torus actions on symplectic manifolds}, volume~93 of {\em
  Progress in Mathematics}.
\newblock Birkh\"auser Verlag, Basel, revised edition, 2004.

\bibitem{barvinok-novik-centrally-symmetric}
A.~Barvinok and I.~Novik.
\newblock A centrally symmetric version of the cyclic polytope.
\newblock {\em Discrete Comput. Geom.}, 39(1-3):76--99, 2008.

\bibitem{berger-geometry-I}
M.~Berger.
\newblock {\em Geometry. {I}}.
\newblock Universitext. Springer-Verlag, Berlin, 1994.
\newblock Translated from the 1977 French original by M. Cole and S. Levy,
  Corrected reprint of the 1987 translation.

\bibitem{biliotti-ghigi-2}
L.~Biliotti and A.~Ghigi.
\newblock {S}atake-{F}urstenberg compactifications, the moment map and
  $\lambda_1$.
\newblock  {\em Amer. J. Math.}, 135, no. 1, 237-274, 2013.

\bibitem{borel-ji-libro}
A.~Borel and L.~Ji.
\newblock {\em Compactifications of symmetric and locally symmetric spaces}.
\newblock Mathematics: Theory \& Applications. Birkh\"auser Boston Inc.,
  Boston, MA, 2006.

\bibitem{bourguignon-li-yau}
J.-P. Bourguignon, P.~Li, and S.-T. Yau.
\newblock Upper bound for the first eigenvalue of algebraic submanifolds.
\newblock {\em Comment. Math. Helv.}, 69(2):199--207, 1994.

\bibitem{caratheodory-Koeffizienten}
C.~Carath\'eodory.
\newblock {\"Uber den Variabilit\"atsbereich der Koeffizienten von
  Potenzreihen, die gegebene Werte nicht annehmen.}
\newblock {\em Math. Ann.}, 64:95--115, 1907.

\bibitem{duistermaat-kolk-varadarajan}
J.~J. Duistermaat, J.~A.~C. Kolk, and V.~S. Varadarajan.
\newblock Functions, flows and oscillatory integrals on flag manifolds and
  conjugacy classes in real semisimple {L}ie groups.
\newblock {\em Compositio Math.}, 49(3):309--398, 1983.

\bibitem{gichev-polar}
V.~M. Gichev.
\newblock Polar representations of compact groups and convex hulls of their
  orbits.
\newblock {\em Differential Geom. Appl.}, 28(5):608--614, 2010.

\bibitem{guillemin-sternberg-convexity-1}
V.~Guillemin and S.~Sternberg.
\newblock Convexity properties of the moment mapping.
\newblock {\em Invent. Math.}, 67(3):491--513, 1982.

\bibitem{guillemin-sternberg-techniques}
V.~Guillemin and S.~Sternberg.
\newblock {\em Symplectic techniques in physics}.
\newblock Cambridge University Press, Cambridge, second edition, 1990.

\bibitem{heckman-thesis}
G.~Heckman.
\newblock {\em Projection of orbits and asymptotic behaviour of multiplicities
  of compact Lie groups}.
\newblock 1980.
\newblock PhD thesis.

\bibitem{heinzner-schwarz-stoetzel}
P.~Heinzner, G.~W. Schwarz, and H.~St{\"o}tzel.
\newblock Stratifications with respect to actions of real reductive groups.
\newblock {\em Compos. Math.}, 144(1):163--185, 2008.

\bibitem{huckleberry-DMV}
A.~Huckleberry.
\newblock Introduction to group actions in symplectic and complex geometry.
\newblock In {\em Infinite dimensional {K}\"ahler manifolds ({O}berwolfach,
  1995)}, volume~31 of {\em DMV Sem.}, pages 1--129. Birkh\"auser, Basel, 2001.

\bibitem{kirillov-lectures}
A.~A. Kirillov.
\newblock {\em Lectures on the orbit method}, volume~64 of {\em Graduate
  Studies in Mathematics}.
\newblock American Mathematical Society, Providence, RI, 2004.

\bibitem{knapp-beyond}
A.~W. Knapp.
\newblock {\em Lie groups beyond an introduction}, volume 140 of {\em Progress
  in Mathematics}.
\newblock Birkh\"auser Boston Inc., Boston, MA, second edition, 2002.

\bibitem{kostant-quantization}
B.~Kostant.
\newblock Quantization and unitary representations. {I}. {P}requantization.
\newblock In {\em Lectures in modern analysis and applications, {III}}, pages
  87--208. Lecture Notes in Math., Vol. 170. Springer, Berlin, 1970.

\bibitem{kostant-convexity}
B.~Kostant.
\newblock On convexity, the {W}eyl group and the {I}wasawa decomposition.
\newblock {\em Ann. Sci. \'Ecole Norm. Sup. (4)}, 6:413--455 (1974), 1973.

\bibitem{mcduff-salamon-symplectic}
D.~McDuff and D.~Salamon.
\newblock {\em Introduction to symplectic topology}.
\newblock Oxford Mathematical Monographs. The Clarendon Press Oxford University
  Press, New York, second edition, 1998.

\bibitem{moore-compactifications}
C.~C. Moore.
\newblock Compactifications of symmetric spaces.
\newblock {\em Amer. J. Math.}, 86:201--218, 1964.

\bibitem{sanyal-sottile-sturmfels-orbitopes}
R.~Sanyal, F.~Sottile, and B.~Sturmfels.
\newblock Orbitopes.
\newblock {\em Mathematika}, 57:275--314, 2011.

\bibitem{satake-compactifications}
I.~Satake.
\newblock On representations and compactifications of symmetric {R}iemannian
  spaces.
\newblock {\em Ann. of Math. (2)}, 71:77--110, 1960.

\bibitem{schneider-convex-bodies}
R.~Schneider.
\newblock {\em Convex bodies: the {B}runn-{M}inkowski theory}, volume~44 of
  {\em Encyclopedia of Mathematics and its Applications}.
\newblock Cambridge University Press, Cambridge, 1993.

\bibitem{smilansky-1985}
Z.~Smilansky.
\newblock Convex hulls of generalized moment curves.
\newblock {\em Israel J. Math.}, 52(1-2):115--128, 1985.

\end{thebibliography}
\end{document}